\documentclass[11pt,letterpaper]{article}

	\title{Likelihood Asymptotics of Stationary Gaussian Arrays}

\author{Carsten H. Chong\thanks{Department of Information Systems, Business Statistics and Operations Management, The Hong Kong University of Science and Technology, e-mail: carstenchong@ust.hk}
\and Fabian Mies\thanks{Delft Institute of Applied Mathematics, Delft University of Technology, e-mail: f.mies@tudelft.nl}}
\date{}

\usepackage{amsmath, amsfonts, amsthm, verbatim,
	amssymb,dsfont,color, accents}
	\usepackage{todonotes}
\usepackage[utf8x]{inputenc}		

\usepackage[normalem]{ulem}

\usepackage{tikz}
\usetikzlibrary{arrows.meta,arrows,decorations.pathreplacing}

\usepackage[round]{natbib}

\usepackage{setspace}
\def\T{T}

\usepackage[colorlinks,citecolor=blue,linkcolor=blue,urlcolor=blue]{hyperref}

\theoremstyle{plain}
\newtheorem{theorem}{Theorem} 
\newtheorem{lemma}{Lemma}

\newtheorem{corollary}{Corollary}

\theoremstyle{definition}

\newtheorem{assumption}{Assumption}
\theoremstyle{remark}
\newtheorem{remark}{Remark}

\newcommand{\N}{\mathbb{N}}
\newcommand{\Z}{\mathbb{Z}}

\newcommand{\R}{\mathbb{R}}

\newcommand{\E}{\mathbb{E}}

\newcommand{\eps}{\epsilon}

\newcommand{\pconv}{\plim}

\newcommand{\Var}{\operatorname{Var}}

\newcommand{\plim}{\stackrel{\P}{\longrightarrow}}
\newcommand{\wconv}{\stackrel{d}{\longrightarrow}}

\newcommand{\tr}{\mathrm{tr}}

\newcommand{\balpha}{\boldsymbol{\alpha}}

\renewcommand{\P}{\mathbb{P}}

\newcommand{\bone}{\mathbf 1}

\newcommand{\calh}{{\mathcal H}}

\newcommand{\cala}{{\mathcal A}}

\newcommand{\calp}{{\mathcal P}}

\newcommand{\caln}{{\mathcal N}}

\newcommand{\al}{{\alpha}}
\newcommand{\la}{{\lambda}}
\newcommand{\La}{{\Lambda}}

\newcommand{\ga}{{\gamma}}
\newcommand{\Ga}{{\Gamma}}
\newcommand{\vp}{{\varphi}}
\newcommand{\si}{{\sigma}}

\newcommand{\ov}{\overline}
\newcommand{\un}{\underline}

\newcommand{\Den}{\Delta_n}

\newcommand{\op}{\mathrm{op}}

\allowdisplaybreaks[4]

\makeatletter
\newcommand{\settheoremtag}[1]{
	\let\oldtheassumption\theassumption
	\renewcommand{\theassumption}{#1}
	\g@addto@macro\endassumption{
		\global\let\theassumption\oldtheassumption}
}
\makeatother

\definecolor{darkblue}{RGB}{44, 107, 172}
\definecolor{darkred}{RGB}{132,28,28}

\numberwithin{equation}{section} 

\begin{document}

	\maketitle

	\begin{abstract}
		This paper develops an asymptotic likelihood theory for triangular arrays of stationary Gaussian time series depending on a multidimensional unknown parameter. 
		We give sufficient conditions for the associated sequence of statistical models to be locally asymptotically  normal in \mbox{Le Cam}'s sense, 
		which in particular implies the asymptotic efficiency of the maximum likelihood estimator. 
		Unique features of the array setting covered by our theory include potentially nondiagonal rate matrices as well as spectral densities that satisfy different power-law bounds at different frequencies and    may fail to be uniformly integrable.
		To illustrate our theory, we study  efficient estimation for Gaussian processes sampled at high frequency and for a class of  autoregressive models with moderate deviations from a unit root.
	\end{abstract}
	
\bigskip
\noindent \emph{Keywords:}
	Fractional Brownian motion;  high-frequency data;	local asymptotic normality; local-to-unity model; maximum likelihood estimator; mildly integrated process; Toeplitz matrix
		
\section{Introduction}

Maximum likelihood estimation  possesses attractive asymptotic properties: it is well known that the maximum likelihood estimator (MLE) based on independent and identically distributed random variables is consistent and asymptotically efficient under mild regularity assumptions (\cite{LC98}). In a time series context, the asymptotic properties of the MLE has also been extensively studied in the literature. Consider, for example, a
univariate stationary centered Gaussian time series $(X_t)_{t=1}^n$ with spectral density $f^\theta$ depending on an $M$-dimensional unknown parameter $\theta\in\Theta$. Classical results by 
\cite{Dahlhaus1989} and \cite{Lieberman2012} state that under mild assumptions on $f^\theta$, the MLE of $\theta$ is consistent, asymptotically efficient and normal with covariance matrix $n^{-1}I(\theta_0)^{-1}$, where $I(\theta)$ is the  asymptotic Fisher information matrix
\begin{equation}
	I(\theta)= \frac{1}{4\pi}\int_{-\pi}^{\pi} \frac{\nabla f^{\theta}(\lambda) \nabla f^{\theta}(\lambda)^\T }{f^{\theta}(\lambda)^2}\, d\lambda \label{eqn:fisher-sequence}
\end{equation}
and $\theta_0$ is the true value of the unknown parameter $\theta$.  \cite{Cohen2013}  further showed that the associated sequence of statistical models $(\P^n_\theta)_{n\in\N,\theta\in\Theta}$ has the \emph{local asymptotic normality (LAN)} property in Le Cam's sense (cf.\ also \cite{TYZ25}). That is, for every $\theta\in\Theta$, there is a sequence of random vectors $(\xi_n(\theta))_{n\in\N}$ such that as $n\to\infty$, 
\begin{equation}\label{eq:LAN} 
	\xi_n(\theta)\stackrel{d}{\longrightarrow} \caln(0,I(\theta)),\quad \log \frac{d\P^n_{\theta+R_na}}{d\P^n_\theta} - a^\T\xi_n(\theta)+\frac12a^\T I(\theta)a \stackrel{\P}{\longrightarrow}0\quad (a\in\R^M)
\end{equation}
under $\P^n_\theta$, where $R_n=n^{-1/2}I_M$ and $I_M$ is the identity matrix in $\R^M$. The LAN property  implies the asymptotic efficiency of the MLE (see Section II.12 of \cite{IH81}):  For any sequence of estimators $\widehat{\theta}_n$ and any symmetric nonnegative quasi-convex loss function $\ell$ with $\ell(z) = o( \exp(\epsilon |z|^2))$ as $|z|\to\infty$ for any $\epsilon>0$, 
\begin{align}
	\liminf_{\delta\to\infty} \liminf_{n\to\infty} \sup_{|R_n^{-1}(\theta-\theta_0)|\leq \delta} \E (\ell(R_n^{-1}(\widehat{\theta}_n-\theta))) 
	\geq \E (\ell(Z)), \quad Z \sim \caln(0, I(\theta_0)^{-1}). \label{eqn:Lecam}
\end{align}

In this paper, we aim to derive an analogous result to \eqref{eq:LAN} for \emph{arrays} of stationary Gaussian time series. This will be the content of our main result, Theorem~\ref{thm:LAN} below. 
Compared to the classical time series framework, there are two distinct challenges when arrays of time series are studied.

First, the spectral density $f_n^\theta$ varies with the sample size and, in  various  examples of interest, in such a way that the right-hand side in \eqref{eqn:fisher-sequence} with $f^\theta_n$ substituted for $f^\theta$ converges to a \emph{non-invertible} limit as $n\to\infty$  (\cite{Kawai2013}). In these situations, the asymptotic efficiency of the MLE no longer follows from the  degenerate LAN result. To recover non-degeneracy of the limiting Fisher information matrix, we resort to \emph{nondiagonal} rate matrices $R_n$ in \eqref{eq:LAN}. Indeed, even in cases where the limiting Fisher information matrix is singular upon the classical choice $R_n=n^{-1/2}I_M$, it is often possible to obtain a regular limit by choosing nondiagonal rate matrices $R_n$. This in turn guarantees   asymptotic efficiency of the MLE and determines its (optimal) asymptotic variance.
The fact that the limiting Fisher information matrix may fail to be regular and the possibility of restoring invertibility by taking nondiagonal rate matrices were first observed by \cite{brouste2018local}, in the particular case of estimating fractional Brownian motion  under infill asymptotics.

Second, we need to devise a new set of regularity conditions concerning the behavior of the spectral densities $f^\theta_n$ around zero that allows us to prove Theorem~\ref{thm:LAN} and, at the same time, applies to a broad variety of typical examples.  To this end, we require the spectral density and its inverse  and  derivatives be expressible as linear combinations of functions each of which satisfies multiple power-law bounds around zero with different exponents. In addition, we explicitly allow the sequence of spectral densities $f^\theta_n$ to be not uniformly integrable, with a pole of an order larger than $1$ at the origin in the limit. Structural assumptions of this generality are needed in order to cover the examples we study and constitute substantial generalizations of the conditions required in the classical time series setting (see e.g., \cite{Lieberman2012}), which bound the spectral density and its inverse and derivatives by a single integrable power-law function. 

We illustrate our theoretical results in three typical situations where time series arrays arise. First, we consider estimation based on high-frequency data. Because the number of data points increases on a fixed interval, this is a typical case where the observations form a time series array. In Corollary~\ref{cor:2}, we show    the LAN property for a superposition of two independent fractional Brownian motions with different Hurst parameters. \cite{chong_when_2022} applies an extension of this model to the modeling of microstructure noise in high-frequency asset return data. Our result implies that the estimators of \cite{chong_when_2022} attain an optimal rate of convergence in general. As in the case of a single fractional Brownian motion, the rate matrices must be nondiagonal in this example  to guarantee an invertible limiting Fisher information matrix.

Next, we show that our theory also applies to   estimation under joint high-frequency and long-span asymptotics. In Corollary~\ref{cor:1}, we investigate the LAN property and the asymptotic efficiency of the MLE for a stationary fractional Ornstein--Uhlenbeck (fOU) process. This class of processes attracted increased attention recently in the econometric analysis of rough volatility (see e.g., \cite{WXY23,WXYZ24} and  \cite{SYZ24}). Even in the   simple case of a classical OU process, the associated spectral densities   are not uniformly integrable, a challenge that we address in our theory. The results we obtain for fOU processes further extend those obtained by, for instance, \cite{A02} and \cite{G02}, to a non-Markovian setting.


Finally, we apply   our theory  to time series with a sample-size dependent parameter. More precisely, we
consider a first-order autoregressive (AR) process where the AR coefficient  $\vp_n$ may change with the sample size. Besides the classical stationary case $\vp_n\equiv\vp \in(-1,1)$, both the cases $\vp_n\equiv 1$ (unit root) and $\vp_n=1-c/n$ with $c>0$ (local to unit root) have been extensively studied in the statistics and econometrics literature. We refer to \cite{DF79,DF81}, \cite{CW87}, \cite{P87} and \cite{PP88} for early seminal works and to \cite{Phillips2018} for a review of this topic. Against this background, our theory is applicable to  a class of \emph{mildly integrated} autoregressive processes (see \cite{PM07} and \cite{P23}), which have roots of the form $\vp_n=1-c/n^\al$ for some $c>0$ and $\al\in(0,1)$ and bridge the classical stationary and local-to-unity regimes. Such models have found applications in the econometric modeling of imploding price bubbles (\cite{PS18}) and in IVX estimation (\cite{PM09}, \cite{KMS14}, \cite{CZZ24}). Our paper contributes to this literature by showing, in Corollary~\ref{cor:3}, the LAN property  for mildly integrated processes with  Gaussian shocks for any $\al\in(0,\frac15)$, which implies the asymptotic efficiency of the MLE.  As the MLE has the same asymptotic distribution as the OLS estimator studied in \cite{PM07} and \cite{P23}, our theory also proves the asymptotic efficiency of the latter, if $\al\in(0,\frac15)$. 
Whether the LAN property (and asymptotic efficiency of the MLE) holds for $\al\in[\frac15,1)$ remains an open problem.

The remainder of the paper is organized as follows: Section~\ref{sec:LAN} contains the main result of the paper. Section~\ref{sec:ex} contains three examples that our theory is applicable to. 
 In Appendix~\ref{app:reg}, we  present more  general assumptions under which our main theorem holds. In Appendix~\ref{sec:examples}, we detail the proofs of Corollaries~\ref{cor:2}--\ref{cor:3}. 
 Appendix~\ref{app:technical} contains technical results needed for  the proof of our main result, which is given in Appendix~\ref{sec:proofs}.

\section{LAN Property of Stationary Gaussian Arrays}\label{sec:LAN}
For  $n\in\N$,  consider a centered Gaussian stationary time series $(X^{(n)}_t)_{t=1}^n$ with spectral density  $$f_n^\theta(\la)=\frac{1}{2\pi}\biggl( \E[(X^{(n)}_1)^2]+2\sum_{k=1}^\infty \cos(k\la)\E[X^{(n)}_1X^{(n)}_{1+k}]\biggr),\qquad\la\in(-\pi,\pi),$$
where $\theta\in\Theta$ is the unknown parameter of interest in this model and the parameter space $\Theta$ is an open subset of $\R^M$. It is well known that $X_n=(X^{(n)}_1,\dots,X^{(n)}_n)^\T \sim \caln(0, T_n(f^\theta_n))$, where $T_n(f)= (\int_{-\pi}^\pi e^{i(k-j)\la}f(\la)\,d\la)_{k,j=1}^n$ denotes the  Toeplitz matrix in $\R^{n\times n}$ associated to   an integrable symmetric function $f$. The assumptions in this paper guarantee that $f_n^\theta$ is strictly positive, which implies that $T_n(f_n^\theta)$ is positive definite and that the log-likelihood function  and the Fisher information matrix 
in this model are   
\begin{align*}
	l_n(\theta) &= -\frac{n}{2} \log (2\pi) -\frac{1}{2} \log [\det T_n(f_n^{\theta})] - \frac{1}{2}   X_n^\T T_n(f_n^{\theta})^{-1} X_n ,\\
	\widetilde{I}_n(\theta) 
	& = \frac{1}{2} \left( \tr\left( T_n(f_n^\theta)^{-1} T_n (\partial_{j} f_n^\theta) T_n(f_n^\theta)^{-1} T_n (\partial_{k} f_n^\theta) \right) \right)_{j,k=1}^M\\ 
	& \approx I_n(\theta)= \frac{n}{4\pi}\left( \int_{-\pi}^\pi \frac{\partial_j f_n^\theta(\lambda)  \partial_k f_n^\theta(\lambda)}{f_n^\theta(\lambda)^2} \,d\lambda\right)_{j,k=1}^M,
\end{align*} 
respectively.
The  approximation in the last line is known to hold asymptotically for fixed spectral densities $f_n^\theta=f^\theta$, and a major part of our proofs is concerned with this approximation for arrays. 
Our qualitative finding is that, under suitable regularity conditions, all information about the optimal rates of convergence and variances is contained in the pre-asymptotic Fisher information matrix $I_n(\theta)$, see \eqref{eq:cond1.1} below.

Given $h: \Theta\to \R$ and a multi-index $\beta=(\beta_1,\dots,\beta_M)^\T\in \N_0^M$, we write $
\partial_\beta h(\theta) = \partial_1^{\beta_1}\cdots \partial_M^{\beta_M}  h(\theta)$ and $|\beta| = \sum_{k=1}^M \beta_k$. We use $\lVert\cdot\rVert$ to denote an arbitrary norm on $\R^M$ or $\R^{M\times M}$. We also assume that the true parameter value $\theta_0\in\Theta$ is fixed throughout.
\begin{assumption}[Information]\label{ass:gauss-lan-1}
	There exist regular rate matrices $R_n\in \R^{M\times M}$ and  $I(\theta_0)\in\R^{M\times M}$ such that  for some small $\epsilon>0$ and all $\delta>0$, we have $n^\eps\|R_n\|\to0$ and
	\begin{equation}\label{eq:cond1.1} 
		\sup_{\|\theta-\theta_0\| \leq n^{-\delta}} \left\|
		R_n^\T I_n(\theta) R_n  - I(\theta_0)\right\|  \to 0
	\end{equation}
	as $n\to\infty$.
	Moreover, for any sufficiently small $\eta>0$, there are $K(\eta)>0$ and an open neighborhood $B(\eta)\subset \Theta$ of $\theta_0$ such that for all $\theta\in B(\eta)$, $|\beta|\leq 1$ and sufficiently large $n$,
	\begin{equation}\label{eq:cond1.2} 
		\int_{-\pi}^{\pi} |\lambda|^{-\eta} \frac{\|\partial_\beta R_n^\T \nabla f_n^{\theta}(\lambda)\|^2}{f_n^\theta(\lambda)^2}\, d\lambda 
		+ \int_{-\pi}^{\pi} |\lambda|^{-\eta} \frac{\|\partial_\beta R_n^\T D^2 f_n^{\theta}(\lambda) \|^2}{f_n^\theta(\lambda)^2}\, d\lambda \leq K n^{-1+\eta}.
	\end{equation}
\end{assumption}

Condition \eqref{eq:cond1.1} guarantees the existence of a regular limiting Fisher information matrix, while \eqref{eq:cond1.2} imposes mild integrability conditions on the higher-order derivatives of $f^\theta_n$. For classical stationary time series where $f_n^\theta=f^\theta$ does not depend on $n$,  one may choose the parametric scaling matrix $R_n = n^{-1/2} {I}_M$, in which case \eqref{eq:cond1.1} becomes trivial. In an array setting, however, choosing a nondiagonal rate matrix may be necessary for $I(\theta_0)$ to be invertible, see Corollaries~\ref{cor:2} and \ref{cor:1} below.

We also need growth and regularity bounds on the spectral density around zero. Letting $\Ga$   be the family of symmetric   integrable  functions on $\La=(-\pi,\pi)\setminus\{0\}$, we express these conditions through the following classes of functions, in which $c> 0$, $\alpha\in \R$ and $L:(0,\infty)\to (0,\infty)$:
\begin{equation}
	\begin{split}
		\Gamma_0(c,\alpha,L) &= \{ h\in\Gamma :   |h(\lambda)| \leq c |\lambda|^{-\alpha-\epsilon} L(\epsilon),\ \la\in\La,\ \epsilon>0  \}, \\
		\Gamma_1(c,\alpha,L) &= \{ h\in\Gamma\cap C^1(\La):   |h(\lambda)| + |\lambda \, h'(\lambda)| \leq c |\lambda|^{-\alpha-\epsilon} L(\epsilon),\ \la\in\La,\ \epsilon>0  \}.
	\end{split}
	\label{eqn:rate-form}
\end{equation}
The function $L$ allows for   slowly varying  terms in the small-$\la$ asymptotics of $h$. If an infinite collection of functions is considered, it also ensures uniformity of the bounds in \eqref{eqn:rate-form}. 

\begin{assumption}[Regularity]\label{ass:gauss-lan-2} For some $m, q\in\N$ and for all $|\beta|=0,1,2,3$, we can write 
	\begin{equation}
		\begin{aligned}
			\partial_\beta f_{n}^\theta &=\sum_{i=1}^m f^\theta_{n,i,\beta} &&\text{with }f^\theta_{n,i,\beta} \in \bigcap_{k=1}^q\Gamma_1(c_{i,k}(n,\theta), \alpha_{i,k}(\theta), L),\\
			\frac{1}{f_n^\theta} &= \sum_{k=1}^q \ov f_{n,k}^\theta &&\text{with }\ov f_{n,k}^\theta \in \bigcap_{i=1}^m \Gamma_1 ( 1/\ov c_{i,k}(n,\theta), -\alpha_{i,k}(\theta),L  ),
		\end{aligned}
		\label{eqn:f-form}
	\end{equation}
	for some $L:(0,\infty)\to(0,\infty)$, and the coefficients satisfy the following properties:
	For any $\eta>0$, there exist    $K=K(\eta)\in(0,\infty)$ and an open neighborhood $B(\eta)\subset \Theta$ of $\theta_0$ such that for all $i=1,\dots,m$, $k=1,\dots,q$, $n\in\N$, and $\theta,\theta'\in B(\eta)$, we have
	\begin{equation}\label{eq:cond2.1} 
		-1<\al_{i,k}(\theta)<1,\quad |\alpha_{i,k}(\theta)-\alpha_{i,k}(\theta')| \leq \eta,\quad	\frac{c_{i,k}(n,\theta)}{\ov c_{i,k}(n,\theta)} \leq Kn^\eta,\quad\frac{c_{i,k}(n,\theta)}{c_{i,k}(n,\theta')} \leq K n^\eta.
	\end{equation}
\end{assumption}

If $m=q=1$ and $c_{i,k}(n,\theta) = c_{i,k,\beta}(n,\theta) = 1$, we recover Assumptions 1--3 of \cite{Lieberman2012} (augmented by the conditions from  \cite{takabatake_note_2022}). 
Condition \eqref{eqn:f-form} is met, for instance, by functions satisfying $f_n^\theta(\lambda)\approx (\sum_{i=1}^m a_{i}(n) |\lambda|^{\alpha_i})/(\sum_{k=1}^q b_{k}(n) |\lambda|^{\gamma_k})$ for $\lambda$ close to zero and $a_i(n),b_i(n)>0$, in which case we have $\alpha_{i,k}(\theta)=\gamma_k-\alpha_i$ and $\ov c_{i,k}(n,\theta) = a_i(n)/b_k(n)$. We encounter a case where $q=2$ and $m=1$ in 
 Corollary \ref{cor:2} below, while the case  $q=1$ and $m=3$ arises in both Corollaries~\ref{cor:1} and \ref{cor:3}.

The first condition in \eqref{eq:cond2.1} ensures  that $f_n^\theta$ is integrable around zero, while the other three assert some forms of boundedness and continuity of the coefficients and exponents of $f_n^\theta$ and its derivatives. Assumption~\ref{ass:gauss-lan-2}, albeit quite general already, is too restrictive for many interesting situations (including those discussed in Corollaries~\ref{cor:1} and \ref{cor:3}), as the spectral densities in these cases  fail to be uniformly integrable and violate the requirement from \eqref{eq:cond2.1} that $\al_{i,k}(\theta)<1$ for all $i$ and $k$. To cover such processes, we present  an extension of Assumption~\ref{ass:gauss-lan-2} that allows for exponents bigger than $1$ in Appendix~\ref{app:reg}.

We are now in the position to state the main theorem of the paper, which generalizes the results of  \cite{hannan_asymptotic_1973}, \cite{Fox1986}, \cite{K87}, \cite{Dahlhaus1989}, \cite{Lieberman2012} and \cite{Cohen2013} to a triangular array setting. 

\begin{theorem}\label{thm:LAN}
	Suppose that Assumption \ref{ass:gauss-lan-1} and either Assumption \ref{ass:gauss-lan-2} or Assumption~\ref{ass:gauss-lan-2-ext} from Appendix~\ref{app:reg} hold for $\theta_0\in\Theta$.
	Then there exists a sequence of random vectors $\xi_n(\theta_0)$ such that  under $\P^n_{\theta_0}$,  
	\begin{equation}\label{eq:thm.1} \begin{split}
			&\xi_n(\theta_0) \stackrel{d}{\longrightarrow} \mathcal{N}(0, I(\theta_0)),\\	&\sup_{a\in\mathcal{H}}\left|l_n(\theta_0 + R_n a) - l_n(\theta_0) - \langle a, \xi_n(\theta_0)\rangle + \frac{1}{2} a^\T I(\theta_0) a \right| \stackrel{\P}{\longrightarrow} 0\end{split}
	\end{equation}
	for any compact set $\mathcal{H}\subset \R^M$.
	Hence, the model $\P^n_\theta$ is LAN in $\theta_0$ with rate matrix $R_n$ and asymptotic Fisher information matrix $I(\theta_0)$.
	Moreover, there exists a sequence of random vectors $\widehat{\theta}_n\in\Theta$ satisfying $\P^n_{\theta_0}(\nabla l_n(\widehat{\theta}_n)=0)\to 1$ such that   under $\P^n_{\theta_0}$ and as $n\to\infty$,
	\begin{equation}\label{eq:thm.2} 
		R_n^{-1} (\widehat{\theta}_n - \theta_0)  \stackrel{d}{\longrightarrow}  \mathcal{N}(0, I(\theta_0)^{-1}).
	\end{equation}
	The estimator $\widehat\theta_n$ is asymptotically efficient in the sense of \eqref{eqn:Lecam}. 
\end{theorem}

The proof is given in Appendix~\ref{sec:proofs}. It uses a lemma concerning Toeplitz matrices that was first stated and proved in \cite{taqqu_toeplitz_1987} (p.~80) and \cite{Dahlhaus1989} (Lemma~5.3). This lemma, however, is erroneous, leaving a gap also in the proof of the results for the classical time series setting obtained by \cite{Lieberman2012} and \cite{Cohen2013}. We close this gap by providing a corrected version in Lemma~\ref{lem:toeplitz-norm-0} of the appendix.

\section{Applications}\label{sec:ex}

We illustrate the use of Theorem~\ref{thm:LAN} in three different situations where triangular arrays of time series naturally arise. In all cases, Theorem~\ref{thm:LAN} allows us to derive new results on the asymptotics of the MLE and its efficiency. All proofs appear in Appendix~\ref{sec:examples}. 

\subsection{High-Frequency Asymptotics}

The first situation concerns estimation of a Gaussian process under an infill asymptotic setup. 
We have an array of time series in this setting
because the number of observations increases to infinity on finite time interval.
For illustration purposes, we consider estimation of a  \emph{mixed fractional Brownian motion} based on high-frequency observations
\begin{equation}\label{eq:X2} 
	X_{i\Den}= \si_1 B^{H_1}_{i\Den} + \si_2 B^{H_2}_{i\Den},\quad  i=0,\dots, \lfloor T/\Den \rfloor,
\end{equation}
where $B^{H_1}$ and $B^{H_2}$ are two independent fractional Brownian motions with Hurst parameters $H_1$ and $H_2$, respectively. The unknown parameter in this model is   $\theta=(H_1,\si_1^2, H_2, \si_2^2)^\T$  and we assume that it belongs to
\begin{equation}\label{eq:Theta2} 
	\Theta=\Bigl\{ (H_1,\si_1^2, H_2, \si_2^2)^\T : 0<H_1<H_2<1,\ H_2-H_1 <\tfrac14,\ \si_1^2,\si_2^2\in(0,\infty) \Bigr\}.
\end{equation}
The length of the interval, $T$, is fixed.
To ease notation, we denote the true value of $\theta$ by $\ov\theta=(\ov H_1,\ov\si_1^2,\ov H_2,\ov \si_2^2)^\T$, instead of adding a subscript $0$.


Model \eqref{eq:X2} was introduced by \cite{Cheridito2001} and is called a \emph{mixed fractional Brownian motion}; see also \cite{chong_when_2022,chong_rate-optimal_2022} and \cite{mies_estimation_2023} for extensions of this model. 
If $  H_2-H_1  >\frac14$, neither $H_2$ nor $\si_2^2$ can be consistently estimated from high-frequency data on a finite time interval (\cite{VanZanten2007}). This is similar to the fact that a drift cannot be consistently estimated on a finite interval in the presence of a diffusive component   and the reason why we require $H_2-H_1<\frac14$ in \eqref{eq:Theta2}. 


\begin{corollary}\label{cor:2}
	For any $\ov\theta\in\Theta$, the conclusions of Theorem~\ref{thm:LAN} hold true for the distribution $P^n_\theta$ of $(\Delta^n_i X)_{i=1}^{\lfloor T/\Den\rfloor} $ with
	\begin{equation}\label{eq:R1}
		R_n	=\begin{pmatrix} R_n(\ov\si^2_1)&0\\ 0 & R_n(\ov\si^2_2)\Den^{-2(\ov H_2-\ov H_1)} \end{pmatrix}, \quad R_n(\si^2)=\begin{pmatrix} \Den^{1/2} & 0   \\ 2 \si^2 \Den^{1/2}\log \Den^{-1}&\Den^{1/2}  \end{pmatrix}
	\end{equation}
	and  
	\begin{equation}\label{eq:I1} 
			I(\theta)=\frac{T}{4\pi\si_1^4}\int_{-\pi}^\pi  \begin{pmatrix}v_{\si^2_1, H_1}(\la) \\ v_{\si^2_2, H_2}(\la) \end{pmatrix} \begin{pmatrix}v_{\si^2_1, H_1}(\la) \\ v_{\si^2_2, H_2}(\la) \end{pmatrix}^\T f_{ H_1}(\la)^{-2}\,d\la, \quad v_{\si^2,H}(\la) = \begin{pmatrix} \si^2\partial_Hf_H(\la) \\ f_H(\la)\end{pmatrix},
	\end{equation}
where $f_H$ is the spectral density of the increments of fractional Brownian motion:
	\begin{equation}\label{eq:fH} 
	f_H(\la) = C_H(1-\cos(\la))\sum_{k\in\Z} \lvert 2\pi k+\la\rvert^{-1-2H},\quad 	C_H=\frac{\Ga(2H+1)\sin(\pi H)}{\pi}.
\end{equation}
\end{corollary}

As in \cite{brouste2018local}, it is important to choose a nondiagonal rate matrix in this example to guarantee that the limiting Fisher information is regular and that the LAN property holds. The corollary also shows that the estimators considered in \cite{chong_when_2022} in a semiparametric setting that corresponds to $H_1<\frac12$ and $H_2=\frac12$ are rate-optimal in the special case of a mixed fractional Brownian motion.


\subsection{High-Frequency and Long-Span Asymptotics}

Next, we demonstrate how Theorem~\ref{thm:LAN} also applies  under joint high-frequency and long-span asymptotics. To be   specific, we consider the  stationary  solution to the fractional OU equation
	\begin{equation}\label{eq:OU} 
	dX_t = -\kappa X_t dt + \si dB^H_t,\quad 
\end{equation}
where $\theta=(\kappa,H,\si)^\T \in \Theta= (0,1)\times(0,\infty)^2$, $B^H$ is a fractional Brownian motion with Hurst parameter $H$ and $X_0\sim N(0,\frac{\si^2}{\kappa^{2H}}H\Ga(2H))$. We assume that we observe $X_{i\Delta_n}$, $i = 0,\dots, \lfloor T/\Den\rfloor$, where $\Delta_n \to0$ and $T=T_n\sim C\Den^{-\beta}\to\infty$ as $n\to\infty$ for some $C,\beta>0$.
Fractional OU processes have been considered in \cite{WXY23,WXYZ24} and \cite{SYZ24}, among others, for the purpose of modeling volatility in continuous time.

\begin{corollary}\label{cor:1} If $\theta_0\in\Theta$ and $1+\frac14\beta -5H+2H^2>0$, the conclusions of Theorem~\ref{thm:LAN} hold true for the distribution $\P^n_\theta$ of $(X_{i\Den})_{i=1}^{\lfloor T/\Den\rfloor}$ with
	\begin{equation}\label{eq:R2}\begin{split}
			R_n	&= \begin{pmatrix}  T^{-1/2} &0&0\\ 0& (T/\Den)^{-1/2} & 0 \\ 0 & 2\si^2(T/\Den)^{-1/2}\log \Den^{-1} & (T/\Den)^{-1/2} \end{pmatrix},
	\\
		I(\theta)&= \begin{pmatrix} 1/(2\kappa) & 0\\
			0&\frac1{4\pi}\displaystyle\int_{-\pi}^\pi \begin{pmatrix} F_H(\la)^2 & F_H(\la)\si^{-2}\\ F_H(\la)\si^{-2}&\si^{-4}\end{pmatrix}d\la  \end{pmatrix},\end{split}
	\end{equation}
where
\begin{equation}\label{eq:FG} 
	F(H)=\frac{C'(H)}{C(H)} +\frac{2\log \lvert\la\rvert^{-1}+2\lvert \la\rvert^{2H+1}\partial_H G_H(\la)}{1+2\lvert \la\rvert^{2H+1}G_H(\la)},\quad G_H(\la) = \sum_{k=1}^\infty \lvert \la + 2\pi k\rvert^{-1-2H}
\end{equation}
and the constants are given by $C(H)=\Ga(2H+1)\sin(\pi H)$ and $C'(H) = \partial_H C(H)$.
\end{corollary}


If $H\leq \frac{5-\sqrt{17}}4\approx 0.2192$, the condition $1+\frac14\beta -5H+2H^2>0$ is automatically satisfied for $\beta>0$. If $H>\frac{5-\sqrt{17}}4\approx 0.2192$, the condition imposes a minimum speed, relative to $\Den^{-1}$, at which the time horizon $T$ must increase to infinity. We need this condition due to the fact that even for a process as simple as a stationary fractional OU process, the spectral density under joint infill and long-span asymptotics is not uniformly integrable. Therefore, we need the extended regularity conditions of Assumption~\ref{ass:gauss-lan-2-ext} and the restriction on $T$ and $\Den$ is needed to satisfy the last condition in \eqref{eq:cond2.1-2}. At present, we do not know whether this condition is sharp in general. In the special case $H=\frac12$, 
\cite{G02} shows the LAN property with no other  restrictions on $T$ and $\Den$ than $T\to\infty$ and $\Den\to0$. The proof of \cite{G02}, however, utilizes the Markov property of the solution, which does not extend to fractional OU processes with $H\neq\frac12$. Maximum likelihood estimation for diffusion processes has also been studied by \cite{A02}.

\subsection{Sample-Size Dependent Parameter}

If the unknown parameter of a time series depends on the sample size, the resulting model is formally a time series array. We show how Theorem~\ref{thm:LAN} can be used to establish the LAN property for a class of local-to-unity processes. 
Let $(\eps_t)_{t\in\N}$ be independent standard normal random variables and suppose that 
\begin{equation}\label{eq:AR1} 
	X_t^{(n)}=(1-ca_n)X_{t-1}^{(n)} + \si\eps_t,\quad n\in\N,\ t=1,\dots,n,
\end{equation}
where $(a_n)_{n\in\N}$ is a null sequence such that $na_n\to\infty$ and $X_0^{(n)}\sim \caln(0,\frac{\si^2}{1-\phi_n^2})$ is independent of $(\eps_t)_{t\in\N}$. Here,  $\phi_n=1-ca_n$. By construction, $(X_t^{(n)})_{t\geq0}$ is a stationary $\mathrm{AR}(1)$ process with autoregressive coefficient $\phi_n$, which converges to $1$ slower than $\frac1n$. The unknown parameter in \eqref{eq:AR1} is $\theta=(c,\si^2)\in\Theta=(0,\infty)^2$ and we denote the true value of $\theta$  by $\theta_0=(c_0,\si^2_0)$.

The process $X^{(n)}_t$ is called a \emph{mildly integrated} $\mathrm{AR}(1)$ process  (see \cite{PM07} and \cite{P23}) and bridges the stationary and unit root regimes of autoregressive processes. The restriction $na_n\to\infty$ is important: if $a_n=\frac1n$, the resulting local-to-unity models have been extensively studied in the literature. In particular, it is known  in this case that $c$ cannot be consistently estimated and the MLE for $c$ is neither asymptotically normal (\cite{CW87}, \cite{P87}) nor efficient (\cite{HPS11}, \cite{P12}). If $a_n^{-1}$ increases to infinity not too fast, Assumption \ref{ass:gauss-lan-2-ext} holds and Theorem~\ref{thm:LAN} can be used to derive the LAN property and the asymptotic efficiency of the MLE.

\begin{corollary}\label{cor:3}
	Suppose that $a_n^{-1}=O(n^\al)$ for some $\al\in(0,\frac15)$. If $\theta_0\in\Theta$, the conclusions of Theorem~\ref{thm:LAN} hold true for the distribution $\P^n_\theta$ of $(X_{t}^{(n)})_{t=0}^n$, with
	\begin{equation}\label{eq:RI3}
		R_n=\begin{pmatrix} (na_n)^{-1/2} & 0\\ 0&n^{-1/2}\end{pmatrix},\qquad	I(\theta)=\begin{pmatrix}c^{-1}/2&0\\0&\si^{-4}/2\end{pmatrix}.
	\end{equation}
	The OLS estimator of $c$ from \cite{PM07} has the same asymptotic distribution as the MLE of $c$ and is therefore asymptotically efficient.
\end{corollary}

The restriction on $\al$  is a consequence of the last condition in \eqref{eq:cond2.1-2} of Assumption \ref{ass:gauss-lan-2-ext}.
If $a_n^{-1}$ increases to infinity at speed $n^{1/5}$ or faster, but slower than $n$, it remains an open question whether the LAN property holds and whether the MLE is asymptotically efficient. 



\appendix

\section{Extended Regularity Assumptions}\label{app:reg}
 
 We give an extension of Assumption~\ref{ass:gauss-lan-2} that allows for exponents bigger than $1$.
 \settheoremtag{2E}
 \begin{assumption}[Extended Regularity]\label{ass:gauss-lan-2-ext} There are $m, q\in\N$ such that 
 	\begin{equation}
 		\begin{aligned}
 			\partial_\beta f_{n}^{\theta} &=\sum_{i=1}^m  f^\theta_{n,i,\beta} &&\text{with }f^\theta_{n,i,\beta}\in  \bigcap_{k=1}^q  \Gamma_1(c_{i,k}(n,\theta), \alpha_{i,k}(\theta), L),	 \quad\text{for }|\beta|=0,1,2,3,\\
 			\frac{1}{f_n^\theta} &\in\sum_{k=1}^q \ov f^\theta_{n,k}&&\text{with } \ov f^\theta_{n,k}\in \bigcap_{i=1}^m  \Gamma_1 ( 1/\ov c_{i,k}(n,\theta), -\ov \alpha_{i,k}(\theta),L  ),
 		\end{aligned}\label{eqn:f-form-2}
 	\end{equation}
 	where   $c_{i,k}(n,\theta),\ov c_{i,k}(n,\theta)>0$, $\al_{i,k}(\theta),\ov\al_{i,k}(\theta)\in \R$ and $L:(0,\infty)\to(0,\infty)$ is a given function. We assume that there is a choice of indices $i'=i'(i,k)$,  $i^\ast=i^\ast(i,k)$,  $k'=k'(i,k)$ and  $k^\ast=k^\ast(i,k)$ with the following property:
 	For any sufficiently small $\eta\in(0,1)$ and some $r>0$ and $\iota\in(0,1)$, there exist    $K=K(\eta,r,\iota)\in (0,\infty)$ and an open neighborhood $B(\eta,r,\iota)\subset \Theta$ of $\theta_0$ such that for all $i=1,\dots,m$, $k=1,\dots,q$, $n\in\N$, $\lvert\beta\rvert=1,2,3$ and $\theta,\theta'\in B(\eta,r,\iota)$, we have  
 	\begin{equation}\label{eq:cond2.1-2} \begin{split}
 			&\al_{i,k'}(\theta)=\ov\al_{i',k}(\theta), \qquad \al_{i,k^\ast}(\theta)\leq \ov\al_{i^\ast,k}(\theta),\qquad \al_{i,k^\ast} (\theta)<1,\qquad \ov\al_{i^\ast,k}(\theta)>-1, \\
 			&\min_k \ov \al_{i',k}(\theta)<1,\qquad |\alpha_{i,k}(\theta)-\alpha_{i,k}(\theta')| \leq \eta, \qquad\frac{1}{\sqrt{Kn^{r}}}\leq\ov c_{i,k}(n,\theta)\leq \sqrt{Kn^{r}},\\
 			&\frac{c_{i,k'}(n,\theta)}{\ov c_{i',k}(n,\theta)} \leq Kn^\eta,\qquad\frac{c_{i,k}(n,\theta)}{c_{i,k}(n,\theta')} \leq K n^\eta,\qquad \frac{c_{i,k^\ast}(n,\theta)}{\ov c_{i^\ast,k}(n,\theta)} \leq K\lVert R_n\rVert^{-1/2+\iota}.
 		\end{split}
 	\end{equation}
 \end{assumption}
 
 The added generality of Assumption~\ref{ass:gauss-lan-2-ext} over Assumption~\ref{ass:gauss-lan-2} is that we allow some exponents $\alpha_{i,k}(\theta)$ to be larger than $1$. 
 If all exponents are less than $1$, we almost recover Assumption~\ref{ass:gauss-lan-2} by setting $i'=i^\ast=i$, $k'=k^\ast=k$ and $\al_{i,k}(\theta)=\ov\al_{i,k}(\theta)$.
 The only difference is that we do not assume the growth condition $\frac{1}{\sqrt{Kn^{r}}}\leq\ov c_{i,k}(n,\theta)\leq \sqrt{Kn^{r}}$ in Assumption~\ref{ass:gauss-lan-2}.
 
 If some of the exponents are bigger than $1$, the   reason why one might not be able to choose $i'=i^\ast=i$, $k'=k^\ast=k$ and $\al_{i,k}(\theta)=\ov\al_{i,k}(\theta)$ is because we need $\al_{i,k^\ast}(\theta)<1$. For example, if for some $k$, we have $\ov\al_{i,k}(\theta)>1$ for all $i$, then we are forced to choose $k^\ast$ in such a way that $k^\ast\neq k$ and $\al_{i,k^\ast}(\theta)<\ov\al_{i^\ast,k}(\theta)$. In this case, the most restrictive condition in \eqref{eq:cond2.1-2} is the last one, which requires the ratio of coefficients resulting from choosing $k^\ast \neq k$ or $i^\ast\neq i$ be marginally smaller than $\lVert R_n\rVert^{-1/2}$.

\section{Proof of Corollaries~\ref{cor:2}--\ref{cor:3}}\label{sec:examples}

\begin{proof}[Proof of Corollary~\ref{cor:2}]
	The spectral density of $(\Delta^n_i X)_{i=1}^{\lfloor T/\Den\rfloor}$ is given by 
	\begin{equation}\label{eq:fn1} 
		f_n^\theta(\la) = \si_1^2 \Delta_n^{2H_1} f_{H_1}(\la)+\si_2^2 \Delta_n^{2H_2} f_{H_2}(\la)
	\end{equation}
	at stage $n$. Moreover, we have 
	\begin{equation}\label{eq:grad}\nabla f_n^\theta(\la)=\begin{pmatrix}
			\si_1^2\Den^{2H_1} (\partial_H f_{H_1}(\la)-2(\log \Den^{-1})f_{H_1}(\la))\\ \Den^{2H_1}f_{H_1}(\la),\\
			\si_2^2\Den^{2H_2} (\partial_H f_{H_2}(\la)-2(\log \Den^{-1})f_{H_2}(\la))\\ \Den^{2H_2}f_{H_2}(\la)
	\end{pmatrix}\end{equation}
	and $D^2f_n^\theta(\la)=(D^2f_n^\theta(\la)_{ij})_{i,j=1}^4$ with
	\begin{equation}\label{eq:Hessian} \begin{split}
			D^2f_n^\theta(\la)_{11}&=  \si_1^2\Den^{2H_1}[\partial_H^2 f_{H_1}(\la)-4(\log \Den^{-1})\partial_Hf_{H_1}(\la)+4(\log \Den^{-1})^2f_{H_1}(\la)],\\
			D^2f_n^\theta(\la)_{12}&=D^2f_n^\theta(\la)_{21}=\Den^{2H_1} (\partial_H f_{H_1}(\la)-2(\log \Den^{-1})f_{H_1}(\la)),\\
			D^2f_n^\theta(\la)_{13}&=D^2f_n^\theta(\la)_{31}=D^2f_n^\theta(\la)_{14}=D^2f_n^\theta(\la)_{41}=0,\\
			D^2f_n^\theta(\la)_{22}&=D^2f_n^\theta(\la)_{23}=D^2f_n^\theta(\la)_{32}=D^2f_n^\theta(\la)_{24}=D^2f_n^\theta(\la)_{42}=0,\\
			D^2f_n^\theta(\la)_{33}&=\si_2^2\Den^{2H_2}[\partial_H^2 f_{H_2}(\la)-4(\log \Den^{-1})\partial_Hf_{H_2}(\la)+4(\log \Den^{-1})^2f_{H_2}(\la)],\\
			D^2f_n^\theta(\la)_{34}&=D^2f_n^\theta(\la)_{43}=\Den^{2H_2} (\partial_H f_{H_2}(\la)-2(\log \Den^{-1})f_{H_2}(\la)),\\
			D^2f_n^\theta(\la)_{44}&=0.
		\end{split}
	\end{equation}
	
	We verify Assumptions~\ref{ass:gauss-lan-1} and \ref{ass:gauss-lan-2} with  rate matrix given by \eqref{eq:R1}
	and Fisher information matrix given by \eqref{eq:I1}.
	We start with \eqref{eq:cond1.1} and note that
		\begin{equation}	\label{eq:fH0} 
		\partial^k_H f_H(\la) \sim 2^{k-1}C_H\lvert \la\rvert^{1-2H}(\log_\vee \lvert\la\rvert^{-1})^{k},\quad \text{as } \la\to0.
	\end{equation} 
By  \eqref{eq:grad}, \eqref{eq:fH0} and the fact that $\Den^{1/2+2(H_2-\ov H_2)+2\ov H_1}-\Den^{1/2+2H_1}=O(\Den^{1/2+2H_1+\eta}\log \Den^{-1})$ for $\|\theta-\ov\theta\| \leq \Den^{\eta}$, we have
	\begin{align*}
		R_n^\T\nabla f_n^\theta(\la)	&=\begin{pmatrix} \si_1^2\Den^{2H_1+1/2}\partial_H f_{H_1}(\la)+O(\Den^{\eta+2H_1+1/2}(\log\Den^{-1})\lvert\la\rvert^{1-2H_1})\\
			\Den^{2H_1+1/2}f_{H_1}(\la)\\
			\si_2^2\Den^{2H_1+1/2}\partial_H f_{H_2}(\la) + O(\Den^{\eta+2H_1+1/2}(\log \Den^{-1})\lvert \la\rvert^{1-2H_2}\log_\vee \lvert\la\rvert^{-1})\\
			\Den^{2H_1+1/2}f_{H_2}(\la)+ O(\Den^{\eta+2H_1+1/2}(\log \Den^{-1})\lvert \la\rvert^{1-2H_2})
		\end{pmatrix}\\
		&=\begin{pmatrix} \si_1^2\partial_H f_{H_1}(\la) \\
			f_{H_1}(\la)\\
			\si_2^2 \partial_H f_{H_2}(\la)  \\
			f_{H_2}(\la) 
		\end{pmatrix}\Den^{2H_1+1/2}+ O(\Den^{\eta+2H_1+1/2}(\log \Den^{-1})\lvert \la\rvert^{1-2H_2}\log_\vee \lvert\la\rvert^{-1})\\
		&=\begin{pmatrix} \ov\si_1^2\partial_H f_{\ov H_1}(\la) \\
			f_{\ov H_1}(\la)\\
			\ov	\si_2^2 \partial_H f_{\ov H_2}(\la)  \\
			f_{\ov H_2}(\la) 
		\end{pmatrix}\Den^{2H_1+1/2}+ O(\Den^{\eta+2H_1+1/2}(\log \Den^{-1})\lvert \la\rvert^{1-2\ov H_2-2\eta}(\log_\vee \lvert\la\rvert^{-1})^2)
	\end{align*}
	for all $\theta$ satisfying $\|\theta-\ov\theta\| \leq \Den^{\eta}$.
	At the same time, 
	\begin{align*}
		\frac{\Den^{4H_1}}{f_n^\theta(\la)^2}	&=\frac{1}{(\si_1^2f_{H_1}(\la)+\si_2^2\Den^{2(H_2-H_1)} f_{H_2}(\la))^2}\\
		&=\frac{1}{\si_1^4 f_{H_1}(\la)^2} - \frac{2\si_1^2f_{H_1}(\la) \si_2^2\Den^{2(H_2-H_1)}f_{H_2}(\la)+(\si_2^2\Den^{2(H_2-H_1)}f_{H_2}(\la))^2}{(\si_1^2f_{H_1}(\la))^2(\si_2^2\Den^{2(H_2-H_1)}f_{H_2}(\la))^2},
	\end{align*}
	which implies
	\[ \biggl\lvert\frac{\Den^{4H_1}}{f_n^\theta(\la)^2}	-\frac{1}{\si_1^4 f_{H_1}(\la)^2} \biggr\rvert \leq \frac{4(\si_2^2\Den^{2(H_2-H_1)}f_{H_2}(\la))^\eps}{(\si_1^2f_{H_1}(\la))^{2+\eps}} \]
	for any $\eps\in(0,1)$.
	Therefore, 
	\begin{align*}
		&\frac{\lfloor T/\Den\rfloor}{4\pi}\int_{-\pi}^\pi \frac{(R_n^\T\nabla f_n^\theta(\la))(R_n^\T\nabla f_n^\theta(\la))^\T}{f_n^\theta(\la)^2}d\la\\
		&\quad=I(\ov\theta)+O\Biggl( \int_{-\pi}^\pi \frac{\Den^\eta(\log\Den^{-1})\lvert \la\rvert^{2-4\ov H_2-2\eta}(\log_\vee \lvert \la\rvert^{-1})^3}{f_{H_1}(\la)^2} d\la \Biggr)\\
		&\quad\quad+O\Biggl( \int_{-\pi}^\pi \frac{\lvert \la\rvert^{2-4\ov H_2-4\eta}(\log_\vee \lvert\la\rvert^{-1})^4(\Den^{2(H_2-H_1)}f_{H_2}(\la))^\eps}{f_{H_1}(\la)^{2+\eps}} d\la\Biggr) \\
		&\quad=O\Biggl( (\Den^{\eta}(\log\Den^{-1})\vee \Den^{2\eps(\ov H_2-\ov H_1-2\eta)})\int_{-\pi}^\pi \lvert\la\rvert^{-(4+2\eps)(\ov H_2-\ov H_1)-(8+4\eps)\eta}(\log_\vee \lvert \la\rvert^{-1})^4 d\la \Biggr). 
	\end{align*}
	Because $\ov H_1<\ov H_2$ and $\ov H_2-\ov H_1<\frac14$ by assumption and there is no loss of generality to assume that $\eta>0$ is small, the last integral is finite for sufficiently small $\eps$ and $\eta$, showing \eqref{eq:cond1.1}. To prove \eqref{eq:cond1.2}, we can use \eqref{eq:grad} and \eqref{eq:Hessian} in conjunction with \eqref{eq:fn1} and \eqref{eq:fH0} to obtain the bounds
	\begin{align*}
		\frac{\|\partial_\beta R_n^\T \nabla f_n^{\theta}(\lambda)\|^2}{f_n^\theta(\lambda)^2}&\leq K_\theta \Biggl( \frac{\Den^{1/2+2\ov H_1-4\eta} (\log \Den^{-1})^{1+\lvert \beta\rvert} \lvert\la\rvert^{1-2\ov H_2-2\eta}(\log_\vee \lvert\la\rvert^{-1})^{1+\lvert \beta\rvert}}{\Den^{2\ov H_1+2\eta}\lvert\la\rvert^{1-2\ov H_1+2\eta}} \Biggr)^2
		\\
		&= K_\theta \Den^{1-12\eta}(\log \Den^{-1})^{2+2\lvert \beta\rvert} \lvert \la\rvert^{-4(\ov H_2-\ov H_1)-8\eta}(\log_\vee \lvert\la\rvert^{-1})^{2+2\lvert \beta\rvert},\\
		\frac{\|\partial_\beta R_n ^\T D^2 f_n^{\theta}(\lambda) \|^2}{f_n^\theta(\lambda)^2}&\leq K_\theta \Biggl( \frac{\Den^{1/2+2\ov H_1-4\eta} (\log \Den^{-1})^{2+\lvert \beta\rvert} \lvert\la\rvert^{1-2\ov H_2-2\eta}(\log_\vee \lvert\la\rvert^{-1})^{2+\lvert \beta\rvert}}{\Den^{2\ov H_1+2\eta}\lvert\la\rvert^{1-2\ov H_1+2\eta}} \Biggr)^2
		\\
		&= K_\theta \Den^{1-12\eta}(\log \Den^{-1})^{4+2\lvert \beta\rvert} \lvert \la\rvert^{-4(\ov H_2-\ov H_1)-8\eta}(\log_\vee \lvert\la\rvert^{-1})^{4+2\lvert \beta\rvert},
	\end{align*}
	for all $\theta\in\Theta$ with $\lVert \theta-\ov\theta\rVert\leq \eta$ and some constant $K_\theta$ that varies continuously in $\theta\in\Theta$. As $\ov H_2-\ov H_1<\frac14$ and $\eta>0$ was arbitrary, this implies \eqref{eq:cond1.2}.
	
	In order to verify Assumption~\ref{ass:gauss-lan-2}, we choose $m=2$ and $q=1$ and decompose $f_n^\theta=f_{n,1}^\theta+f_{n,2}^\theta$ where $f_{n,i}^\theta(\la)=\si_i^2\Den^{2H_i}f_{H_i}(\la)$ for $i=1,2$.  By \eqref{eq:fH0}, there is
	$K>1$   such that
	\[ \lvert\partial_\beta f_H(\la)\rvert + \lvert \la \partial_\beta f'_H(\la)\rvert \leq K C_H\lvert \la\rvert^{1-2H}(\log_\vee \lvert\la\rvert^{-1})^{\lvert\beta\rvert} \]
	and 
	\begin{equation}\label{eq:bound}
		\lvert\partial_\beta f_{n,i}^\theta(\la)\rvert	+\lvert\la\partial_\beta (f_{n,i}^\theta)'(\la)\rvert	\leq KC_{H_i}(\si_i^2\vee1)\Den^{2H_i}(\log \Den^{-1})^{\lvert\beta\rvert}\lvert\la\rvert^{1-2H_i} (\log_\vee \lvert\la\rvert^{-1})^{\lvert\beta\rvert} 
	\end{equation}
	for all $\lvert\beta\rvert\leq 3$ and $H\in(0,1)$.
	In fact, one can choose $K$ such that for both $i=1,2$, we also have
	\begin{equation}\label{eq:recip1} 
		\frac{1}{f_n^\theta(\la)}\leq \frac{1}{f_{n,i}^\theta(\la)}\leq \frac{K}{C_{H_i}\si_i^2\Den^{2H_i}}\lvert \la\rvert^{2H_i-1}
	\end{equation}
	and
	\begin{equation}\label{eq:recip2} 
		\biggl\lvert\la\biggl(\frac{1}{f_n^\theta(\la)}\biggr)'	\biggr\rvert\leq \frac{\lvert \la (f_n^\theta)'(\la) \rvert}{f_n^\theta(\la)^2}\leq \frac{\lvert \la\rvert}{f_{n,i}^\theta(\la)}\biggl( \frac{\lvert(f_{n,1}^\theta)'(\la)\rvert}{f_{n,1}^\theta(\la)} + \frac{\lvert(f_{n,2}^\theta)'(\la)\rvert}{f_{n,2}^\theta(\la)} \biggr) \leq \frac{2K^3}{C_{H_i}\si_i^2\Den^{2H_i}} \lvert \la\rvert^{2H_i-1}.
	\end{equation}
	Therefore, defining $L(\eps)= 2K^4(\frac{1}{e\eps^3}\vee \pi^\eps)$ and 
	\begin{equation}\label{eq:choice} 
		c_{i,1}(n,\theta)=KC_{H_i}\si_i^2 \Den^{2H_i},\quad   c_{i,1,\beta}(n,\theta)=KC_{H_i}(\si_i^2\vee1)\Den^{2H_i}(\log \Den^{-1})^{\lvert\beta\rvert}, \quad \al_{i,1}(\theta)=2H_i-1,
	\end{equation}
	for $i=1,2$ and $\lvert\beta\rvert=1,2,3$,
	and noting that $\lvert\la\rvert^\eps(\log_\vee \lvert\la\rvert^{-1})^3 \leq \frac{1}{e\eps^3}\vee \pi^\eps$ uniformly in $\la\in(-\pi,\pi)\setminus\{0\}$,
	we have $f_{n,i}^\theta\in \Ga_1(c_{i,1}(n,\theta),2H_i-1,L)$, $1/f_n^\theta \in \Ga_1(c_{1,1}(n,\theta),1-2H_1,L)\cap\Ga_1(c_{2,1}(n,\theta),1-2H_2,L)$ as well as $\partial_\beta f_{n,i}^\theta\in \Ga_1( c_{i,1,\beta}(n,\theta),2H_i-1,L)$ for $i=1,2$ and $\lvert\beta\rvert=1,2,3$. The choice in \eqref{eq:choice} also satisfies \eqref{eq:cond2.1}. 
	
	Finally, to verify the invertibility of $I(\ov \theta)$, it suffices by Lemma~\ref{lem:posdef} to show that for all $z\in\R^4$,
	\[\begin{pmatrix} \ov\si_1^2\partial_H f_{\ov H_1}(\la) /f_{\ov H_1}(\la)\\
		1\\
		\ov	\si_2^2 \partial_H f_{\ov H_2}(\la) /f_{\ov H_1}(\la)  \\
		f_{\ov H_2}(\la) / f_{\ov H_1}(\la)
	\end{pmatrix}^\T z\neq0 \]
	for some $\la$. To see this, observe that as $\la\to0$, the previous line is asymptotically equivalent to
	\[\begin{pmatrix}2 \ov\si_1^2\log \lvert\la\rvert^{-1}\\
		1\\
		2\ov	\si_2^2 (C_{\ov H_2}/C_{\ov H_1}) \lvert\la\rvert^{2(\ov H_1-\ov H_2)} \log \lvert\la\rvert^{-1}  \\
		(C_{\ov H_2}/C_{\ov H_1}) \lvert\la\rvert^{2(\ov H_1-\ov H_2)}
	\end{pmatrix}^\T z\neq0, \]
	which holds true for sufficiently small $\la$ because each entry of the first vector has a different asymptotic behavior around $\la=0$.
\end{proof}

\begin{proof}[Proof of Corollary~\ref{cor:1}]
	The spectral density of $(X_{i\Den})_{i=0}^{\lfloor T/\Den\rfloor}$ is given by
\begin{align*}
f^\theta_n(\la) &= \frac{\si^2}{2\pi}C(H)\sum_{k\in\Z} \frac{\Den^{2H}\lvert \la+2\pi k\rvert^{1-2H}}{\kappa^2\Den^2+ (\la+2\pi k)^2} \\
&= \frac{\si^2}{2\pi}C(H) \frac{\Den^{2H}\lvert\la\rvert^{1-2H}}{\kappa^2\Den^2+\la^2} + \frac{\si^2}{2\pi}C(H)\Den^{2H} g_n^\theta(\la),
\end{align*}
	where   $g_n^\theta(\la) = \sum_{k\neq 0}\frac{\lvert \la+2\pi k\rvert^{1-2H}}{\kappa^2\Den^2+(\la+2\pi k)^2}$; see \cite{SYZ24}. Note that 
	\[ \nabla f^\theta_n(\la)= \begin{pmatrix} -\frac{\si^2}{2\pi}C(H)\frac{2\kappa\Den^{2+2H}\lvert \la\rvert^{1-2H}}{(\kappa^2\Den^2+\la^2)^2} + \frac{\si^2}{2\pi}C(H)\Den^{2+2H}\kappa h_n^\theta(\la)\\
	\Den^{2H}	\partial_H k_n^\theta(\la) - 2f_n^\theta(\la)\log \Den^{-1}\\
		 f_n^\theta(\la)/\si^2 \end{pmatrix}, \]
	where 
	\[h_n^\theta(\la)= -2 \sum_{k\neq0}\frac{\lvert\la+2\pi k\rvert^{1-2H}}{(\kappa^2\Den^2+(\la+2\pi k)^2)^{2}},\quad k_n^\theta(\la)=f_n^\theta(\la)/\Den^{2H}.\] 
	Therefore, if $\lVert \theta-\theta_0\rVert\leq (T/\Den)^{-\delta}$, 
	\begin{align*}
		(R_n^\T I_n(\theta) R_n)_{11}	&=\frac{\lfloor T/\Den\rfloor}{4\pi T}\int_{-\pi}^\pi \biggl(\frac{ 2\kappa\Den^2\lvert\la\rvert^{1-2H}(\kappa^2\Den^2+\la^2)^{-2}-  \Den^2\kappa h_n^\theta(\la)}{ \lvert\la\rvert^{1-2H}(\kappa^2\Den^2+\la^2)^{-1}+  g_n^\theta(\la)} \biggr)^2d\la\\
		&=\frac{\lfloor T/\Den\rfloor\Den}{4\pi T}\int_{-\pi/\Den}^{\pi/\Den} \biggl(\frac{ 2\kappa (\kappa^2+\nu^2)^{-2}-  \Den^{3+2H} \kappa h_n^\theta(\Den\nu)}{ (\kappa^2+\nu^2)^{-1}+  \Den^{1+2H}g_n^\theta(\Den\nu)} \biggr)^2d\nu.
	\end{align*}
	Because $g_n^\theta(\la)$ and $h_n^\theta(\la)$ are  bounded from below and above, uniformly in $n$, $\theta$ and $\la$, we deduce 
	\begin{align*}
		(R^\T_nI_n(\theta)R_n)_{11}&\to \frac{1}{4\pi}\int_\R \biggl(\frac{2\kappa_0(\kappa_0^2+\nu^2)^{-2}}{(\kappa_0^2+\nu^2)^{-1}}\biggr)^2 d\nu = \frac{\kappa_0^2}{\pi} \int_\R \frac{1}{(\kappa_0^2+\nu^2)^2} d\nu=\frac{1}{2\kappa_0}=I(\theta_0)_{11},
	\end{align*}
uniformly for $\lVert \theta-\theta_0\rVert\leq (T/\Den)^{-\delta}$. 

Next, we study the part of the limiting Fisher information matrix that concerns estimation of $(H,\si^2)$. We write $A_{2:3,2:3}$ for the lower-right $2\times 2$ submatrix of a $3\times 3$ matrix $A$. Then, writing $m_n^\theta(\la)=\Den^{2H}\partial_H k_n^\theta(\la)-2f_n^\theta(\la)\log \Den^{-1}$, we have
\begin{align*}
&(R_n^T I_n(\theta)R_n)_{2:3,2:3}\\
	&\quad\approx\frac{1}{4\pi} \int_{-\pi}^\pi \frac{1}{f_n^\theta(\la)^2} \begin{pmatrix} 1 & 2\si^2\log \Den^{-1}\\ 0 &1\end{pmatrix}\begin{pmatrix}  (m_n^\theta(\la))^2 &m_n^\theta(\la)f_n^\theta(\la)/\si^2\\ m_n^\theta(\la)f_n^\theta(\la)/\si^2 & f_n^\theta(\la)^2/\si^4  \end{pmatrix}\begin{pmatrix} 1 & 0 \\ 2\si^2\log \Den^{-1} &1\end{pmatrix} d\la\\
	&\quad\approx\frac{1}{4\pi} \int_{-\pi}^\pi \frac{1}{f_n^\theta(\la)^2} \begin{pmatrix} 1 & 2\si^2\log \Den^{-1}\\ 0 &1\end{pmatrix}\begin{pmatrix}  m_n^\theta(\la)\Den^{2H}\partial_H k_n^\theta(\la) &m_n^\theta(\la)f_n^\theta(\la)/\si^2\\ \Den^{2H}\partial_H k_n^\theta(\la)f_n^\theta(\la)/\si^2 & f_n^\theta(\la)^2/\si^4  \end{pmatrix}  d\la\\
	&\quad\approx\frac{1}{4\pi} \int_{-\pi}^\pi \frac{1}{f_n^\theta(\la)^2} \begin{pmatrix} (\Den^{2H}\partial_H k_n^\theta(\la))^2 & \Den^{2H}\partial_H k_n^\theta(\la)f_n^\theta(\la)/\si^2 \\ \Den^{2H}\partial_H k_n^\theta(\la)f_n^\theta(\la)/\si^2  & f_n^\theta(\la)^2/\si^4 \end{pmatrix}  d\la
\end{align*}
for $\lVert \theta-\theta_0\rVert\leq (T/\Den)^{-\delta}$, where $\approx$ means that the difference between left and right is of smaller order.
Because
\begin{equation}\label{eq:F}\begin{split}
	\frac{\Den^{2H}\partial_H k_n^\theta(\la)}{f_n^\theta(\la)}	&=\frac{C'(H)}{C(H)} +\frac{2 \lvert \la\rvert^{1-2H}\log \lvert\la\rvert^{-1}(\kappa^2\Den^2+\la^2)^{-1}+\partial_H g_n^\theta(\la) }{\lvert \la\rvert^{1-2H}(\kappa^2\Den^2+\la^2)^{-1} + g_n^\theta(\la)} \\
	&=\frac{C'(H)}{C(H)} +\frac{2\log \lvert\la\rvert^{-1} + \lvert \la\rvert^{2H-1}(\kappa^2\Den^2+\la^2)\partial_H g_n^\theta(\la) }{ 1 +  \lvert \la\rvert^{2H-1}(\kappa^2\Den^2+\la^2)g_n^\theta(\la)} \\
	&\to \frac{C'(H)}{C(H)} +\frac{2\log \lvert\la\rvert^{-1}+2\lvert \la\rvert^{2H+1}\partial_H G_H(\la)}{1+2\lvert \la\rvert^{2H+1}G_H(\la)}=F_H(\la),
\end{split}\end{equation}
uniformly in $\la$,  we obtain
\begin{equation*}
	(R_n^T I_n(\theta)R_n)_{2:3,2:3}\to I(\theta_0)_{2:3,2:3},
\end{equation*}
uniformly for $\lVert \theta-\theta_0\rVert\leq (T/\Den)^{-\delta}$.

For the remaining elements of $I(\theta_0)$, note that by \eqref{eq:F},
\begin{align*}
(R_n^TI_n(\theta)R_n)_{12}	&\approx \frac{1}{4\pi \sqrt{\Den}} \int_{-\pi}^\pi \frac{\Den^{2H} \partial_\kappa f_n^\theta(\la)\partial_H k_n^\theta(\la)}{f_n^\theta(\la)^2} d\la  \approx \frac{1}{4\pi \sqrt{\Den}} \int_{-\pi}^\pi \frac{  \partial_\kappa f_n^\theta(\la) }{f_n^\theta(\la)} F_H(\la)d\la \\
	&\approx \frac{\sqrt{\Den}}{4\pi}\int_{-\pi/\Den}^{\pi/\Den} \frac{-2\kappa\lvert\nu\rvert^{1-2H}(\kappa^2+\nu^2)^{-2} + \Den^{3+2H}\kappa h_n^\theta(\nu\Den)}{\lvert\nu\rvert^{1-2H}(\kappa^2+\nu^2)^{-1} + \Den^{1+2H} g_n^\theta(\nu\Den)}F_H(\nu\Den) d\nu\\
	&\approx-\frac{\sqrt{\Den}\kappa}{2\pi}\int_{-\pi/\Den}^{\pi/\Den} \frac{F_H(\nu\Den)}{\kappa^2 + \nu^2} d\nu = O(\sqrt{\Den}\log \Den^{-1})
\end{align*}
and
\begin{align*}
	(R_n^TI_n(\theta)R_n)_{12}	&\approx \frac{1}{4\pi \sqrt{\Den}} \int_{-\pi}^\pi \frac{  \partial_\kappa f_n^\theta(\la)\partial_{\si^2} f_n^\theta(\la)}{f_n^\theta(\la)^2} d\la  \approx \frac{1}{4\pi \si^2\sqrt{\Den}} \int_{-\pi}^\pi \frac{  \partial_\kappa f_n^\theta(\la) }{f_n^\theta(\la)} d\la = O(\sqrt{\Den}).
\end{align*}
This completes the proof of \eqref{eq:cond1.1}, where $I(\theta_0)$ is regular by Lemma~\ref{lem:posdef}.
	Using similar calculations, one can show that  \eqref{eq:cond1.2} follows from the bound
	\begin{align*}
		\lvert	\partial_\beta  f_n^\theta(\la)\rvert  \leq \frac{K\Den^{2H} \lvert \la\rvert^{1-2H}(\log \Den^{-1})^3 (\log_\vee \lvert \la\rvert^{-1})^3}{\kappa^2\Den^2+\la^2},\quad \lvert \beta\rvert\leq3.
	\end{align*}
	
	Next, we verify Assumption~\ref{ass:gauss-lan-2-ext}. To this end, we recommend going through the corresponding arguments in the proof of Corollary~\ref{cor:3} first, in particular the interpolation argument.  Here we apply a similar reasoning with $m=1$ and $q=3$: with $z\in(H,1)$ to be specified later, we choose   
	\begin{equation}\label{eq:choice-ac-2}\begin{aligned}
			\al_{1,1}(\theta)	&=2H-1, &\al_{1,2}(\theta)&=2H+1,&\al_{1,3}(\theta)&=2H+1-2z, \\
			\ov\al_{1,1}(\theta)	&=2H-1, &\ov\al_{1,2}(\theta)&=2H+1,&\ov\al_{1,3}(\theta)&=2H+1
	\end{aligned}\end{equation}
	and  
	\begin{equation}
	 \begin{aligned}
			c_{1,1}(n,\theta)&=\frac{\si^2C(H)}{\kappa^2\Den^{2-2H}} \lvert \log \Den\rvert^3, &c_{1,2}(n,\theta)&=\si^2C(H)\Den^{2H} \lvert\log \Den\rvert^3, \\
			c_{1,3}(n,\theta)&=\frac{\si^2C(H)}{\kappa^{2z}}\Den^{1-3z+2Hz} \lvert \log \Den\rvert^3,\\
			\ov c_{1,1}(n,\theta)&=\frac{\si^2C(H)}{\kappa^2\Den^{2-2H}}, &\ov c_{1,2}(n,\theta)&=\ov c_{1,3}(n,\theta)=\si^2C(H)\Den^{2H},
		\end{aligned}
	\end{equation}
	 as well as $k'$ and $k^\ast$ as in \eqref{eq:choice-k}, that is, 
	\begin{equation*}\begin{aligned}
			k'(1)	&=1, &k'(2)&=2, &k'(3)&=2, \\
			k^\ast(1)	&=1, &k^\ast(2)&=3, &k^\ast(3)&=3,
	\end{aligned}\end{equation*}
with the abbreviation $k^\ast(k) = k^\ast(1,k)$ and $k'(k) = k'(1,k)$.
	The last condition in \eqref{eq:cond2.1-2} yields the restriction
	\begin{align*}
		  \max_{k=1,2,3}\frac{c_{1,k^*(k)}}{\overline{c}_{1,k}} \leq  C \lvert\log \Den\rvert^3\max\left( 1, \, \Den^{1-3z+2Hz-2H} \right) \leq \|R_n\|^{-1/2+\iota}\qquad \text{for some }\iota>0.
	\end{align*}
	With  $\|R_n\| = O(1/\sqrt{T})=O(\Den^{\beta/2})$ and choosing $z$ arbitrarily close to its lower bound $H$,  we obtain the constraint $ 1-5H+2H^2+\frac{\beta}{4}>0$.
\end{proof}

\begin{proof}[Proof of Corollary~\ref{cor:3}]
	The spectral density $f_n^\theta$ of $(X^{(n)}_t)_{t\geq0}$ is 
	\[ f_n^\theta(\la)=\frac{\si^2}{2\pi}\frac1{\lvert 1-\phi_ne^{i\la}\rvert^2}=\frac{\si^2}{2\pi}\frac1{1-2\phi_n\cos \la +\phi_n^2},\quad \la\in(-\pi,\pi), \]
	with gradient
	\[ \nabla f_n^\theta(\la)=\begin{pmatrix} \frac{\si^2}{2\pi}\frac{2a_n(\phi_n-\cos \la)}{(1-2\phi_n\cos \la +\phi_n^2)^2} \\ \frac{1}{2\pi}\frac1{1-2\phi_n\cos \la +\phi_n^2} \end{pmatrix}. \]
	Therefore,
	\begin{equation*}
		R_n^\T\nabla f_n^\theta(\la)	=\begin{pmatrix} \frac1{(na_n)^{1/2}}\frac{\si^2}{2\pi}\frac{2a_n(\phi_n-\cos \la)}{(1-2\phi_n\cos \la +\phi_n^2)^2} \\ \frac1{n^{1/2}} \frac{1}{2\pi}\frac1{1-2\phi_n\cos \la +\phi_n^2} \end{pmatrix} =\frac{f_n^\theta(\la)}{n^{1/2}}\begin{pmatrix} \frac{2a_n^{1/2}(\phi_n-\cos\la)}{1-2\phi_n\cos \la+\phi_n^2} \\ \si^{-2}\end{pmatrix}
	\end{equation*}
	and $R_n^\T \nabla f_n^\theta(\la)/f_n^\theta(\la)=R_n^\T \nabla \bar f_n^\theta(\la)/\bar f_n^\theta(\la)$, which was analyzed above.
	and therefore, 
	\begin{align*}
		&\frac{n}{4\pi}	\int_{-\pi}^\pi \frac{R_n^\T\nabla   f_n^\theta(\la)(R_n^\T\nabla   f_n^\theta(\la))^\T}{  f_n^\theta(\la)^2}\,d\la= \frac1{4\pi}\int_{-\pi}^\pi\begin{pmatrix} \frac{2a_n^{1/2}(\phi_n-\cos\la)}{1-2\phi_n\cos \la+\phi_n^2} \\ \si^{-2}\end{pmatrix}\begin{pmatrix} \frac{2a_n^{1/2}(\phi_n-\cos\la)}{1-2\phi_n\cos \la+\phi_n^2} \\ \si^{-2}\end{pmatrix}^\T\,d\la \\
		&\quad= \frac1{4\pi}\int_{-\pi}^\pi\begin{pmatrix} \frac{4a_n(1-ca_n-\cos\la)^2}{(2(1-ca_n)(1-\cos\la)+c^2a_n^2)^2} & \frac{2a_n^{1/2}(1-ca_n-\cos\la)}{\si^2(2(1-ca_n)(1-\cos\la)+c^2a_n^2)}\\  \frac{2a_n^{1/2}(1-ca_n-\cos\la)}{\si^2(2(1-ca_n)(1-\cos\la)+c^2a_n^2)} & \si^{-4}\end{pmatrix}\,d\la.
	\end{align*}
	Clearly, the lower right entry equals $\si^{-4}/2$, which converges to $I(\theta_0)_{22}$ as $\theta\to\theta_0$. The upper left entry is given by
	\begin{align*}
		& \frac{a_n}{4\pi}\int_{-\pi/a_n}^{\pi/a_n}  \frac{4a_n(1-ca_n-\cos(a_n\nu))^2}{(2(1-ca_n)(1-\cos(a_n\nu))+c^2a_n^2)^2}  \,d\nu \\
		&\quad\to \frac1\pi\int_{\R}   \frac{c_0^2}{(\nu^2+c_0^2)^2}  \,d\nu = c_0^{-1}/2=I(\theta_0)_{11},
	\end{align*}
	while analogous calculations show that the off-diagonal entries converge to $0$.
	Similar arguments can also be used to show that \eqref{eq:cond1.2} follows from the bounds
	\[ \Bigl\lvert \partial_c^j   f_n^\theta(a_n\nu)\Bigr\rvert  +\Bigl\lvert \partial_{\si^2}\partial_c^j   f_n^\theta(a_n\nu)\Bigr\rvert \leq \frac{K}{a_n^2(\nu^2+1)^{2}},\qquad j=1,2,3. \]
	
	Next, we examine the regularity of $f_n^\theta$. Elementary calculations show that 
	\[f_n^\theta(\la)\leq K \sigma^2 \min((ca_n)^{-2},\; \lambda^{-2}).\]
	Unfortunately, the $\la^{-2}$-singularity at the origin as $n\to\infty$ prevents us from working under Assumption~\ref{ass:gauss-lan-2}. 
	The central idea of our solution is to derive an additional upper bound on $f^\theta_n(\lambda)$ via interpolation: 
	Since $(ca_n)^{2z}\la^{2(1-z)}\leq (ca_n)^2\vee \la^2$ for all $z\in[0,1]$, we have $f^\theta_n(\la)\leq K\si^2((ca_n)^{-2} \wedge \la^{-2})=K\si^2/((ca_n)^2\vee \la^2)\leq K\si^2/((ca_n)^{2z}\la^{2(1-z)})$.
	That is, 
	\[ f_n^\theta(\lambda) \leq K \sigma^2 \min\left( (ca_n)^{-2},\; \lambda^{-2},\; (ca_n)^{-2z} \lambda^{-2(1-z)}  \right). \]
	With this new upper bound, we are able to verify the less restrictive (but more cumbersome) Assumption~\ref{ass:gauss-lan-2-ext}.
	In particular, we have \eqref{eqn:f-form-2} with $m=1$, $q=3$ and 
	\begin{equation}\label{eq:choice-ac}\begin{aligned}
			\al_{1,1}(\theta)	&=0, &\al_{1,2}(\theta)&=2,&\al_{1,3}(\theta)&=2(1-z), \\
			\ov\al_{1,1}(\theta)	&=0, &\ov\al_{1,2}(\theta)&=2,&\ov\al_{1,3}(\theta)&=2, \\
			c_{1,1}(n,\theta)&=\si^2(ca_n)^{-2}, &c_{1,2}(n,\theta)&=\si^2, &c_{1,3}(n,\theta)&=\si^2(ca_n)^{-2z},\\
			\ov c_{1,1}(n,\theta)&=\si^2(ca_n)^{-2}, &\ov c_{1,2}(n,\theta)&=\si^2, &\ov c_{1,3}(n,\theta)&=\si^2
	\end{aligned}\end{equation}
	and (with the abbreviation $k^\ast (k) = k^\ast(1,k)$ and $k'(k)=k'(1,k)$)
	\begin{equation}\label{eq:choice-k}\begin{aligned}
			k'(1)	&=1, &k'(2)&=2, &k'(3)&=2, \\
			k^\ast(1)	&=1, &k^\ast(2)&=3, &k^\ast(3)&=3
	\end{aligned}\end{equation}
	and $L(\eps)\equiv K$ for some large enough constant $K$. In \eqref{eq:choice-ac}, the bound corresponding to $k=3$ is the interpolation bound. We choose $z\in(\frac12,1)$ to ensure that $\al_{1,3}(\theta)<1$, which is important because $k^\ast(2)=k^\ast(3)=3$ and we need $\al_{1,k^\ast}(\theta)<1$ in \eqref{eq:cond2.1-2}. (Note that $k^\ast =2$ is forbidden because $\al_{1,2}(\theta)>1$.) With the choices made in \eqref{eq:choice-ac} and \eqref{eq:choice-k}, the only condition that is not obvious is the last one for $k=2,3$. It is given by the requirement
	\begin{equation}\label{eq:cond-UR}
		(ca_n)^{-2z}\leq K\lVert R_n\rVert^{-1/2+\iota}.
	\end{equation}
	As $\lVert R_n\rVert\leq K(na_n)^{-1/2}$,  by choosing $z$ close enough to $\frac12$ from above, we can deduce \eqref{eq:cond-UR}  from the assumption  that $a_n^{-1}=O(n^{\al})$ for some $\al<  \frac15$.
\end{proof}

\begin{lemma}\label{lem:posdef}
	Suppose that $I$ is an open interval and $A: I \to \R^{m\times n}$, $\la\mapsto A(\la)$, is a continuous mapping such that $\int_I \lvert A(\la)\rvert^2 \,d\la<\infty$. If $\bigcap_{\la\in I}\ker A(\la)^\T=\{0\}$, then 
	$\int_I A(\la)A(\la)^\T \,d\la$ is positive definite.
\end{lemma}
\begin{proof}
	Let $z\in\R^m\setminus\{0\}$. Then $z^\T A(\la)A(\la)^\T z = \| A(\la)^\T z\|^2_2\geq0$ for all $\la$, so by continuity of the mapping $A(\la)$, it suffices to find one value of $\la$ (which may depend on $z$) such that $A(\la)^\T z\neq0$, in order to conclude that $\int_I A(\la)A(\la)^\T \,d\la$ is positive definite. But such a  $\la$ exists because  $\bigcap_{\la\in I}\ker A(\la)^\T=\{0\}$.
\end{proof}

\section{Technical Tools}\label{app:technical}

In the following, we denote by $K$ a constant which may vary from line to line but is always independent of $n$. For $K=K(\epsilon,\ldots, \balpha)$ (where the ellipsis signifies that $K$ may depend on further variables), where $\balpha$ takes values in $\cala\subset(-\infty,\infty)^d$, we say that $K$ is \emph{bounded on compacts in $\balpha\in\cala$ for  sufficiently small $\epsilon$} if for each compact set $\mathcal{K}$ in $\cala$, there exists some $\epsilon_0>0$ such that
\begin{equation}
	\sup_{\balpha\in \mathcal{K}} K(\epsilon,\ldots, \balpha) < \infty \label{eqn:K-bounded}
\end{equation} 
for all $\eps\leq \eps_0$. We omit the mentioning of $\cala$ if $\cala=(-\infty,\infty)^d$. 
To ease notation, we use the following convention: If a function $f$ has the form $f=\sum_{i=1}^m f_i$ with $f_i\in\Ga_1(c_i,\al_i,L)$ for all $i$, we   use the notation $\balpha_f=(\alpha_{1},\ldots, \alpha_{m})^\T\in\R^{m}$ for the exponents of $f$. If there are multiple functions $f$  under consideration, we use the abbreviation $\balpha$ to denote the concatenation of  $\balpha_f$ of all considered functions $f$. This notation will only be used if the set of functions it refers to is evident from the context.

\subsection{Trace Approximations of Toeplitz Matrices}\label{sec:toeplitz}

Our first technical result considers trace approximations of products of Toeplitz matrices and their inverses  and is in the same spirit as Theorem 3 and Lemma 8 in  \cite{Lieberman2012_preprint}. In order to cover the array setting considered in this paper, we  extend their result to functions of the form stipulated in Assumption~\ref{ass:gauss-lan-2-ext}. Because the considered functions   are linear combinations of functions that satisfy multiple power-law bounds around zero, potentially with some powers  less than $-1$,  both the statement and the proof will be  substantially harder. Also, as we intend to apply this result to functions that may vary with $n$,   all asymptotic expressions have to be bounded rather explicitly. 

We use $\oplus$ to denote the direct sum of sets. For example, ``$\bigoplus_{i=1}^m \bigcap_{k=1}^q \Ga_1(\cdot)$'' stands for ``can be decomposed into $m$ functions each of which satisfies $q$ bounds as specified by $\Ga_1(\cdot)$.'' 
\begin{theorem}\label{thm:toeplitz-inverse}
	Suppose that $f_l\geq 0$ and $g_l$, for $l=1,\ldots, p$, are functions such that
	\begin{equation}
		\begin{split}
			f_l &\in \bigoplus_{i=1}^m \bigcap_{k=1}^q\Gamma_1(c_{i,k}, \alpha_{i,k}, L),\qquad 
			\frac{1}{f_l} \in \bigoplus_{k=1}^q \bigcap_{i=1}^m  \Gamma_1 ( 1/\ov c_{i,k}, -\ov \alpha_{i,k},L  ),  \\
			g_l &\in \bigoplus_{i=1}^m \bigcap_{k=1}^q\Gamma_1(d_{i,k}, \beta_{i,k}, L)
		\end{split}
		\label{eqn:fg-inverse-bound}
	\end{equation}
	for some $L:(0,\infty)\to(0,\infty)$, some coefficients $c_{i,k},\ov c_{i,k},d_{i,k}>0$ and exponents $\al_{i,k},\ov\al_{i,k},\beta_{i,k}\in\R$ satisfying $\min_k \al_{i,k}<1$ and $\min_k \beta_{i,k}<1$ for all $i$. 
	Further suppose that $i'=i'(i,k)$, $i^\ast=i^\ast(i,k)$, $i^\star=i^\star(i,k)$, $k'=k'(i,k)$, $k^\ast=k^\ast(i,k)$ and $k^\star=k^\star(i,k)$ 
	are indices, possibly depending on $i$ and $k$, such that for some $\eta\in(0,1)$, we have
	\begin{equation}
		\begin{split}
			&\max_k (\beta_{i,k'}-\ov \al_{i',k})_+ + \min_k \ov \al_{i',k}<1-\eta, ~~	\beta_{i,k'} - \ov \alpha_{i',k}<1-\eta,\\
			&p(\beta_{i,k^\ast} - \ov \alpha_{i^\ast,k})_+<1-\eta,~~p  (\alpha_{i,k^\star} - \ov \alpha_{i^\star,k})_+<1-\eta,\\
			&  \beta_{i,k^\ast} < 1-\eta, ~~  \al_{i,k^\star} <1-\eta,\quad  \ov\al_{i^\ast,k} >-1+\eta,~~   \ov\al_{i^\star,k} >-1+\eta,~~  \al'_{i}<1-\eta
		\end{split}
		\label{eqn:fg-inverse-choice}
	\end{equation}
	for all $i=1,\dots,m$ and $k=1,\dots,q$,	with the notation
	\begin{equation} 
		\al'_{i}=\max_{k} \ov \al_{i',k}\bone_{\{\ov \al_{i',k}<1\}}.
	\end{equation}
	Then there exists a constant $K=K(\epsilon, L, m,p,q,\eta, \balpha)$, uniformly bounded on compacts in $\balpha=(\al_{i,k},\ov \al_{i,k}, \beta_{i,k})_{i,k}$ for    sufficiently small $\epsilon>0$, such that
	\begin{equation}\label{eq:statement} 
		 	\left| \frac{1}{n} \tr \left( \prod_{l=1}^p T_n(g_l) T_n(f_l)^{-1} \right)  - \frac{1}{2\pi} \int_{-\pi}^{\pi} \prod_{l=1}^p \frac{g_l(\lambda)}{f_l(\lambda)} \, d\lambda \right|  	\leq \frac{K}{n} \delta_n^\star \big[(\delta^\ast_{n})^p +  (\delta^\star_{n} \delta_{n})^p\big], 
	\end{equation}
	where
	\begin{equation}\label{eq:deltanl} \begin{split}
			\delta_{n}  &= \sum_{i=1}^{m} \sum_{k=1}^q  \frac{d_{i,k'}}{\ov c_{i',k}} (R_{i,k}n)^{ \frac{ (\beta_{i, k'} - \ov \alpha_{i',k} )_+}{1-(\al'_{i})_+}+\eps},\\
			\delta^\ast_{n} &= \sum_{i=1}^m\sum_{k=1}^q \frac{d_{i,k^\ast}}{\ov c_{i^\ast,k}} n^{(\beta_{i,k^\ast}-\ov \al_{i^\ast,k})_++\eps}, \\
			\delta^\star_{n} &=1\vee \sum_{i=1}^m\sum_{k=1}^q \frac{c_{i,k^\star}}{\ov c_{i^\star,k}} n^{(\alpha_{i,k^\star}-\ov \al_{i^\star,k})_++\eps},
		\end{split}
	\end{equation}
	with  
	\begin{equation}
		\label{eq:R} 
			R_{i,k} = \begin{cases} 1\vee \max\limits_{k_1} \min\limits_{k_2} \frac{\ov c_{i'_2,k_2}}{\ov c_{i'_1,k_1}} &\text{if } \beta_{i,k'}-\ov \al_{i',k}+\eta\geq0,\\
				1&\text{otherwise}.\end{cases}  
	\end{equation}
	In \eqref{eq:R}, the maximum is taken over all $k_1$ satisfying $\ov \al_{i'_1,k_1}+(\beta_{i,k'}-\ov \al_{i',k})\geq1-\eta$ and the minimum is taken over all $k_2$ satisfying $\ov \al_{i'_2,k_2}+(\beta_{i,k'}-\ov \al_{i',k})<1-\eta$, where $i'_1=i'(i,k_1)$ and $i'_2=i'(i,k_2)$. In \eqref{eq:R}, we also use the convention $\max\emptyset = -\infty$.
\end{theorem}

The proof is given in Section \ref{sec:proofs-toeplitz}. 

\begin{remark}\label{rem}
	In Theorem~\ref{thm:toeplitz-inverse}, we only consider the situation where the coefficients $c_{i,k}$, $\al_{i,k}$, $\ov c_{i,k}$, $\ov\al_{i,k}$, $d_{i,k}$ and $\beta_{i,k}$ do not depend on $l$, as this is all we need for the proof of Theorem~\ref{thm:LAN}. The case of $l$-dependent coefficients can be obtained in a similar manner, but at the expense of much heavier notation. 
\end{remark}

\begin{remark}\label{rem:new}
	For the proof of Theorem~\ref{thm:LAN}, we typically apply Theorem~\ref{thm:toeplitz-inverse} to the case where $g_l$ and $f_l$ do not depend on $l$, $f_l=f^n_\theta$ is the spectral density at stage $n$ and $g_l$ is more or less $b_n$ times a derivative of $f_l$. If we ignore $b_n$ (i.e., assume $b_n\equiv 1$) for now, then by Assumption~\ref{ass:gauss-lan-2-ext}, we have
	\[	\frac{d_{i,k'}}{\ov c_{i',k}} =\frac{c_{i,k'}(n,\theta)}{\ov c_{i',k}(n,\theta)} \leq Kn^\eta,\quad \beta_{i,k'} = \al_{i,k'}(\theta)=\ov\al_{i',k}(\theta)=\ov\al_{i',k},\]
	and because $\frac{1}{\sqrt{Kn^{r}}}\leq\ov c_{i,k}(n,\theta)\leq \sqrt{Kn^{r}}$ by Assumption~\ref{ass:gauss-lan-2-ext}, we also have
	\[R_{i,k}\leq Kn^r\]
	for all $i$ and $k$.
	Therefore,  by choosing $\eps<\eta/(1+r)$, we obtain
	\begin{equation}
		\delta_{n} \leq K n^{2\eta}. 
	\end{equation}
This shows that in our applications of Theorem~\ref{thm:toeplitz-inverse}, $\delta_n$ is of the order $b_n$ times an arbitrarily small power of $n$.

This is also true for $\delta^\ast_n$ and $\delta^\star_n$ if the stronger Assumption~\ref{ass:gauss-lan-2} holds. Indeed,  in this case we can simply choose $i'(i,k)=i^\ast(i,k)=i^\star(i,k)=i$ and $k'(i,k)=k^\ast(i,k)=k^\star(i,k)=i$, which implies $R_{i,k}=1$ and consequently $\delta_n=\delta_n^\ast=\delta_n^\star \leq Kn^{2\eta}$. But the situation in general is different under Assumption~\ref{ass:gauss-lan-2-ext}. In this case, with $i^\ast=i^\star$ and $k^\ast=k^\star$, we have the restrictions $\beta_{i,k^\ast}=\al_{i,k^\star}=\al_{i,k^\ast}(\theta)<1$ and $\ov\al_{i^\ast,k}=\ov\al_{i^\star,k}=\ov\al_{i^\ast,k}(\theta)>-1$. Therefore, in general, $\delta^\ast_n$ and $\delta^\star_n$ will \emph{not} be an arbitrarily small polynomial power in $n$.
\end{remark}

\subsection{Central Limit Theorem for Quadratic Forms}\label{sec:quadratic}

In addition to the approximation results for Toeplitz matrices, we need a technical  result for quadratic forms. 
We consider a compact and convex subset  $\Theta_0$  such that $\theta_0\in\Theta_0\subset \Theta$.
For each $\theta\in\Theta_0$, let $f_n^\theta:\La\to (0,\infty)$ be  symmetric, integrable and positive functions (i.e., each $f_n^\theta$ is the spectral density of a stationary time series). 
Moreover, for each $\theta$ and $l=1,\ldots, p$, we have functions $g^\theta_{n,l}$ that are continuously differentiable with respect to $\theta$. Given a function $L:(0,\infty)\to(0,\infty)$, we assume that
\begin{align}
	\begin{split}
		f_n^\theta,\, \partial_r f_n^\theta &\in \bigoplus_{i=1}^m \bigcap_{k=1}^q\Gamma_1(c_{i,k}(n,\theta), \alpha_{i,k}(\theta), L),\\ 
		\frac{1}{f_n^\theta} &\in \bigoplus_{k=1}^q \bigcap_{i=1}^m  \Gamma_1 ( 1/\ov c_{i,k}(n,\theta), -\ov \alpha_{i,k}(\theta),L  ),  \\
		g^\theta_{n,l},\, \partial_r g^\theta_{n,l} &\in \bigoplus_{i=1}^m \bigcap_{k=1}^q\Gamma_1(d_{i,k}(n,\theta), \beta_{i,k}(\theta), L)\qquad (\partial_r=\partial_{\theta_r}),
	\end{split}\label{eqn:ass-clt-fg}
\end{align}
where for all $i$ and $k$, $\theta\in\Theta_0$ and $n\in\N$, we have $c_{i,k}(n,\theta), \ov c_{i,k}(n,\theta), d_{i,k}(n,\theta)>0$ and $\al_{i,k}(\theta),\ov \al_{i,k}(\theta),\beta_{i,k}(\theta)\in\R$ with $\min_k \al_{i,k}(\theta)<1$ and $\min_k \beta_{i,k}(\theta)<1$. Given indices $i'=i'(i,k)$,  $i^\ast=i^\ast(i,k)$, $i^\star=i^\star(i,k)$, $k'=k'(i,k)$,  $k^\ast=k^\ast(i,k)$ and $k^\star=k^\star(i,k)$ as well as $\theta,\theta'\in\Theta_0$, we define
\[ \al'_{i}(\theta)=\max_{k} \ov \al_{i',k}(\theta)\bone_{\{\ov \al_{i',k}(\theta)<1\}}
\]
and
\begin{equation}\begin{split} 
		\delta_{n}(\theta)    &= \sum_{i=1}^{m} \sum_{k=1}^q  \frac{d_{i,k'}(n,\theta)}{\ov c_{i',k}(n,\theta)} (R_{i,k}(n,\theta)n)^{ \frac{ (\beta_{i, k'}(\theta) - \ov\alpha_{i',k}(\theta) )_+}{1-\al'_{i}(\theta)_+}+\eps} ,\\
		\delta^\ast_{n}(\theta) &= \sum_{i=1}^m\sum_{k=1}^q \frac{d_{i,k^\ast}(n,\theta)}{\ov c_{i^\ast,k}(n,\theta)} n^{(\beta_{i,k^\ast}(\theta)-\ov \al_{i^\ast,k}(\theta))_++\eps}, \\
		\delta^\star_{n} (\theta,\theta')&= 1\vee\sum_{i=1}^m\sum_{k=1}^q \frac{c_{i,k^\star}(n,\theta')}{\ov c_{i^\star,k}(n,\theta)} n^{(\alpha_{i,k^\star}(\theta')-\ov \al_{i^\star,k}(\theta))_++\eps},\qquad \delta^\star_{n} (\theta)=\delta^\star_{n} (\theta,\theta),\\
	\end{split}
\end{equation}
where
\begin{equation}
	\label{eq:R-2} 	R_{i,k}(n,\theta) = \begin{cases} 1\vee \max\limits_{k_1} \min\limits_{k_2} \frac{\ov c_{i'(i,k_2),k_2}(n,\theta)}{\ov c_{i'(i,k_1),k_1}(n,\theta)} &\text{if } \beta_{i,k'}(\theta)-\ov \al_{i',k}(\theta)+\eta\geq0,\\
			1&\text{otherwise},\end{cases}\\
\end{equation}
In the previous display, the   maximum is taken over all $k_1$ satisfying $\ov \al_{i'(i,k_1),k_1}(\theta)+(\beta_{i,k'}(\theta)-\ov \al_{i',k}(\theta))\geq1-\eta$, with $\max\emptyset = -\infty$, while the minimum is taken over all $k_2$ satisfying $\ov \al_{i'(i,k_2),k_2}(\theta)+(\beta_{i,k'}(\theta)-\ov \al_{i',k})(\theta)<1-\eta$.

We assume
that there is $r>0$ and  $K=K(r,\eta)>0$ such that
\begin{equation}
	\begin{split}
		& \max_k (\beta_{i,k'}(\theta)-\ov \al_{i',k}(\theta))_+ + \min_k \ov \al_{i',k}(\theta)<1-\eta, \quad	 \beta_{i,k'}(\theta) - \ov \alpha_{i',k}(\theta)<1-\eta,\\
		&p  (\beta_{i,k^\ast}(\theta) - \ov \alpha_{i^\ast,k}(\theta))_+<1-\eta,\quad  p  (\alpha_{i,k^\star}(\theta) - \ov \alpha_{i^\star,k}(\theta))_+<1-\eta,\quad\al'_{i}(\theta) <1-\eta,\\
		&   \beta_{i,k^\ast}(\theta) < 1-\eta, \quad  \ov\al_{i^\ast,k}(\theta) >-1+\eta, \quad  \al_{i,k^\star}(\theta) <1-\eta,\quad    \ov\al_{i^\star,k}(\theta) >-1+\eta,\\
		&   	|\alpha_{i,k}(\theta) - \alpha_{i,k}(\theta') |  \leq \frac{1}{1+r}\eta^2 , \quad 
		\frac{c_{i,k}(n,\theta')}{c_{i,k}(n,\theta)} \leq K n^\eta,\quad
			\frac{	c_{i,k'}(n,\theta)}{\ov c_{i',k}(n,\theta)}\leq Kn^\eta,\quad \al_{i,k'}(\theta)=\ov\al_{i',k}(\theta)
	\end{split}  
\raisetag{4.5\baselineskip}
	\label{eqn:ass-CLT-c}
\end{equation}
for all $\theta,\theta'\in\Theta_0$, $i=1,\dots,m$ and $k=1,\dots,q$. 
These conditions are analogous to the conditions for Theorem \ref{thm:toeplitz-inverse}, but uniformly in $\theta,\theta'$.

	Finally, we require the functions $g^{\theta}_{n,l}$ and their associated coefficients meet the  symmetry conditions
	\begin{equation}\label{eq:symm} 
		g^{\theta}_{n,l} = g^{\theta}_{n,p-l+1}\quad (l=1,\dots,\lfloor \tfrac p2\rfloor).
	\end{equation}
				
	\begin{theorem}\label{thm:CLT}
		Let $X_n\sim \mathcal{N}(0, T_n(f_n^{\theta_0}))$ and suppose that \eqref{eqn:ass-clt-fg}, \eqref{eqn:ass-CLT-c}  and \eqref{eq:symm} hold.
		Consider
		\begin{equation*}
			Z_n(\theta) 
			=   X_n^\T T_n(f^\theta_n)^{-1} \left(\prod_{l=1}^p \left[T_n(g^{\theta}_{n,l}) T_n(f^\theta_n)^{-1} \right] \right)X_n  - \tr \left( \prod_{l=1}^p \left[T_n(g^{\theta}_{n,l}) T_n(f^\theta_n)^{-1} \right] \right)
		\end{equation*}
		and 
		\begin{equation*}
			\phi_n (\theta_0)
			= \left[\frac{n}{\pi}\int_{-\pi}^{\pi} \prod_{l=1}^p \left(\frac{g^{\theta_0}_{n,l}(\lambda)}{f_n^{\theta_0}(\lambda)}\right)^2 \,d\lambda\right]^{\frac{1}{2}}.
		\end{equation*}
		\begin{enumerate}
			\item[(i)] For any $\epsilon>0$, there exists some $K=K(\epsilon, L, m, p,q,\eta,\balpha)>0$, uniformly bounded on compacts in $\balpha$ for sufficiently small $\eps$, such that
			\begin{equation*}
				\|Z_n(\theta_0)\|_{\psi_1} \leq K \phi_n(\theta_0) + K  \delta_n^\star(\theta_0) \Bigl(    \delta^\ast_{n}(\theta_0)^p +\delta^\star_{n}(\theta_0)^{p}   \delta_{n}(\theta_0)^p\Bigr),
			\end{equation*}
			where $\lVert Z\rVert_{\psi_1}=\inf\{t>0: \E[e^{\lvert Z\rvert/t}]\leq 2\}$ denotes the sub-exponential norm.
			\item[(ii)] If $\phi_n(\theta_0)^{-1} \delta_n^\star(\theta_0)   \big( \delta^\ast_{n}(\theta_0)^p +\delta^\star_{n}(\theta_0)^p   \delta_{n}(\theta_0)^p\big)\to0$ for some $\epsilon>0$, then   $\phi_n(\theta_0)^{-1} Z_n(\theta_0) \stackrel{d}{\longrightarrow}\mathcal{N}(0,1)$.
			\item[(iii)] There exists a constant $K=K(\epsilon, L, m, p,q,\eta,\balpha)>0$, uniformly bounded on compacts in $\balpha$ for sufficiently small $\eps$, such that
			\begin{align*}
					\left\|\sup_{\theta\in\Theta_0}\left|Z_n(\theta) - J_n^*(\theta)\right| \right\|_{\psi_1} 
	&	\leq K \sup_{\theta\in\Theta_0} \lVert\theta-\theta_0\rVert\sup_{\theta\in\Theta_0} \Bigl(\phi_n(\theta_0)+  n^{\frac{1}{2}+3\eta}J_n(\theta)\\
	&\quad+n^{5\eta}(1\vee\delta^\star_n(\theta)\lVert \theta-\theta_0\rVert) \delta^\star_n(\theta)^2\big(  \delta^\ast_{n}(\theta)^p+  \delta^\star_n(\theta)^{p}  \delta_{n}(\theta)^p\big)\Bigr),
			\end{align*}
			where 
			\begin{align*}
				J_n(\theta)^2
				&=\int_{-\pi}^\pi |\lambda|^{-3\eta}  \left\{\prod_{l=1}^p \frac{|g^{\theta}_{n,l}(\lambda)| + \| \nabla g^{\theta}_{n,l}(\lambda)\|}{f_n^{\theta}(\lambda)} \right\}^2 \, d\lambda, \\
				J_n^*(\theta)&= \frac{n}{2\pi} \int_{-\pi}^{\pi}  \frac{f_n^{\theta_0}(\lambda) - f_n^\theta(\lambda)}{f_n^\theta(\lambda)} \prod_{l=1}^p \frac{g^{\theta}_{n,l}(\lambda)}{f_n^\theta(\lambda)} \, d\lambda.  
			\end{align*}
		\end{enumerate}
	\end{theorem}
	The proof is given in Section \ref{sec:proofs-quadratic}.

\section{Proofs}\label{sec:proofs}

\subsection{A Key Lemma}\label{sec:key}

One of the main ingredients in the proofs is a lemma concerning the operator norm of Toeplitz matrices that first appeared in \cite{taqqu_toeplitz_1987} (p.\ 80) and \cite{Dahlhaus1989} (Lemma~5.3). This lemma is further used in \cite{Lieberman2012} (Lemma 6 in the preprint \cite{Lieberman2012_preprint}) and \cite{Cohen2013} (Lemma 3.1) to derive the main results of their papers. However, the proofs presented in \cite{taqqu_toeplitz_1987} and \cite{Dahlhaus1989} are erroneous. In this section, we first give a corrected version of this lemma and then extend it to functions satisfying Assumption~\ref{ass:gauss-lan-2}. Our new proof is based on the following auxiliary result.

\begin{lemma}\label{lem:aux-monotone}
	Let $f:(0,\pi) \to (0,\infty)$ be integrable and   $\mathcal{P}_n$ be the class of probability densities on $(0,\pi)$ that are bounded by $n$.
	For any $n\in\N$ and $\kappa>0$, we have 
	\begin{equation}\label{eq:hsup-res}\begin{split}
		&\sup_{h\in \mathcal{P}_n} \frac{\int_0^\pi f(\lambda) \lambda^{-\kappa} h(\lambda)\, d\lambda}{\int_0^\pi f(\lambda) h(\lambda)\, d\lambda} \\
&\quad		\leq \sup\left\{ M \geq 2^{1+\kappa}n^\kappa : n\int_0^{M^{-\frac{1}{\kappa}}} f(\lambda) \lambda^{-\kappa} \, d\lambda >  \frac{M}{4}\inf_{\la\in(\frac1{2n},\pi)} f(\lambda)  \right\} \vee 2^{1+\kappa}n^\kappa, 
		\end{split}
	\end{equation}
	with the convention $\sup\emptyset = 0$.
	For the special case $f(\lambda)=\lambda^{-\alpha}$ with $\alpha<1-\kappa$, the upper bound is $\mathcal{O}( n^{\frac{\kappa}{1-\alpha_+}})$.
\end{lemma}
\begin{proof}
	Let $Q(h)=\frac{\int_0^\pi f(\lambda) \lambda^{-\kappa} h(\lambda)\, d\lambda}{\int_0^\pi f(\lambda) h(\lambda)\, d\lambda}$.
	For any $M>0$, we have 
	\begin{align*}
		Q_n^*:=\sup_{h\in\mathcal{P}_n} Q(h) > M 
		\qquad \iff \qquad \sup_{h\in \mathcal{P}_n} \int_0^\pi f(\lambda) \left[ \lambda^{-\kappa}-M \right] h(\lambda)\, d\lambda> 0. 
	\end{align*}
	Thus, 
	\begin{align}
		Q_n^* = \sup\left\{ M \geq \pi^{-\kappa} : \sup_{h\in \mathcal{P}_n} \int_0^\pi f(\lambda) \left[ \lambda^{-\kappa}-M \right] h(\lambda)\, d\lambda> 0 \right\}. \label{eq:qsup}
	\end{align}
	This formulation is based on a proof idea in   \cite{Lieberman2012_preprint}. 
	Next, observe that the function $\lambda\mapsto f(\lambda)\left[ \lambda^{-\kappa}-M \right]$ is non-negative if and only if $\lambda\leq  M^{-\frac{1}{\kappa}}$. 
	Thus, for any $h\in \mathcal{P}_n$ and $M>2^{1+\kappa} n^{\kappa}$,
	\begin{align*}
	&\int_0^\pi f(\lambda) \left[ \lambda^{-\kappa}-M \right] h(\lambda)\, d\lambda 	\\
	&\quad\leq n\int_0^{M^{-\frac{1}{\kappa}}} f(\lambda) \left[\lambda^{-\kappa} - M\right]\, d\lambda + \int_{\frac{1}{2n}}^\pi h(\lambda)\, d\lambda  \sup_{\lambda\in (\frac{1}{2n}, \pi)} f(\lambda)[\lambda^{-\kappa}-M], 
	\end{align*}
	where we used the bound $\lvert h\rvert\leq n$ for the first integral and the fact that the integral on $(M^{-1/\kappa},\pi)$ is negative and can therefore be shortened to an integral on $(\frac1{2n},\pi)$ (note that $M^{-1/\kappa}<\frac1{2n}$ since $M>2^{1+\kappa}n^\kappa$) to get an upper bound. For any $h\in\calp_n$, we have $\int_{\frac1{2n}}^\pi h(\la)\,d\la \geq \frac12$ (because $\int_0^{\frac1{2n}} h(\la)\,d\la \leq \int_0^{\frac1{2n}} n\,d\la \leq \frac12$). Thus, recalling that the last term in the previous display is negative, we obtain
	\begin{align*}
		\int_0^\pi f(\lambda) \left[ \lambda^{-\kappa}-M \right] h(\lambda)\, d\lambda 	
		&\leq n\int_0^{M^{-\frac{1}{\kappa}}} f(\lambda) \lambda^{-\kappa}\, d\lambda +  \frac12 (2^\kappa n^{\kappa}-M) \inf_{\la\in(\frac1{2n},\pi)} f(\la) \\
		&\leq n\int_0^{M^{-\frac{1}{\kappa}}} f(\lambda) \lambda^{-\kappa} \, d\lambda - \frac{M}{4} \inf_{\la\in(\frac1{2n},\pi)} f(\la).
	\end{align*}
	Using this in \eqref{eq:qsup}, we obtain \eqref{eq:hsup-res}. 
	If $f(\la)=\la^{-\al}$ for some $\al<1-\kappa$,  the right-hand side of \eqref{eq:hsup-res} equals
	\begin{align*}
		&\sup\left\{ M\geq 2^{1+\kappa}n^\kappa: n M^{-\frac{1-\al-\kappa}{\kappa}} >\tfrac M4(1-\al-\kappa)(\pi^{-\al}\bone_{\{\al>0\}} + (2n)^{\al}\bone_{\{\al<0\}}) \right\} \vee 2^{1+\kappa}n^\kappa\\
		&\quad=\left( \left(\frac{1-\al-\kappa}{4} \pi^\al\right)^{-\frac{\kappa}{1-\al}} n^{\frac{\kappa}{1-\al}} \bone_{\{\al>0\}} + \left(\frac{1-\al-\kappa}{2^{2-\al}} \pi^\al\right)^{-\frac{\kappa}{1-\al}}n^\kappa\bone_{\{\al<0\}} \right)\vee 2^{1+\kappa}n^\kappa\\
		&\quad = O\left(n^{\frac{\kappa}{1-\al_+}}\right),
	\end{align*}
	which completes the proof of the lemma.
\end{proof}

The following is a substitute for the lemma on p.\ 80 of \cite{taqqu_toeplitz_1987} and Lemma~5.3 in \cite{Dahlhaus1989}. We denote  the $\ell_2$ operator norm of a matrix by $\|\cdot\|_{\op}$.
\begin{lemma}\label{lem:toeplitz-norm-0}
	Let $g:(-\pi,\pi) \to [0,\infty)$ and $f:(-\pi,\pi) \to (0,\infty)$ be symmetric integrable functions such that $g\in \Gamma_0(d, \beta,L)$ and $1/f \in \Gamma_0(c^{-1}, -\alpha,L)$ for some $c,d>0$, $\al,\beta\in(-\infty,1)$ and $L:(0,\infty)\to(0,\infty)$. Then there exists $K=K(\epsilon,  L, \alpha, \beta)$ that is bounded on compacts in $(\al,\beta)\in(-\infty,1)^2$ for sufficiently small $\eps$ such that 
	\begin{equation}\label{eq:bound-DH} 
		\|T_n(f)^{-\frac{1}{2}} T_n(g)^\frac{1}{2}  \|_{\op}^2 = \| T_n(g)^\frac{1}{2} T_n(f)^{-\frac{1}{2}} \|_{\op}^2  \leq K     \frac dc n^{\frac{(\beta-\alpha)_+}{1-\alpha_+}+\epsilon}.
	\end{equation}
\end{lemma}
\begin{proof}
	The equality in \eqref{eq:bound-DH} holds because the operator norm is invariant under transposition and the matrices are symmetric.
	Moreover, as $T_n(f)$ is invertible (because $f>0$) and both $f$ and $g$ are symmetric, we find that
	\begin{equation}\label{eqn:toeplitz-sup}\begin{split}
		&	\| T_n(g)^\frac{1}{2} T_n(f)^{-\frac{1}{2}} \|_{\op}^2 =\sup_{x\neq 0} \frac{\lVert T_n( g)^\frac{1}{2} T_n(f)^{-\frac{1}{2}}x\rVert_2^2}{\lVert x\rVert_2^2}=\sup_{v\neq 0} \frac{\lVert T_n(g)^\frac{1}{2} v\rVert_2^2}{\lVert T_n(f)^{\frac12}v\rVert_2^2}=\sup_{v\neq0} \frac{ v^\T T_n(g)  v }{v^\T T_n(f) v}  \\
			&\quad=\sup_{\lVert v\rVert=1} \frac{\int_{0}^\pi g(\la)\bigl\lvert \sum_{j=1}^n v_je^{-ij\la}\bigr\rvert^2 \,d\la} {\int_{0}^\pi f(\la)\bigl\lvert \sum_{j=1}^n v_je^{-ij\la}\bigr\rvert^2 \,d\la}\leq \sup_{\lVert v\rVert=1} \frac{\int_{0}^\pi   g(\la)\bigl\lvert \sum_{j=1}^n v_je^{-ij\la}\bigr\rvert^2 \,d\la} {\int_{0}^\pi   f(\la)\bigl\lvert \sum_{j=1}^n v_je^{-ij\la}\bigr\rvert^2 \,d\la}  \\
			&\quad\leq \sup_{h\in\mathcal{P}_n} \frac{\int_{0}^{\pi} g(\lambda) h(\lambda)\, d\lambda}{\int_{0}^{\pi} f(\lambda) h(\lambda)\, d\lambda} 
			\leq \frac{d}{c} L(\delta)^2 \sup_{h\in\mathcal{P}_n} \frac{\int_{0}^{\pi} \lambda^{-\beta - \delta} h(\lambda)\, d\lambda}{\int_{0}^{\pi} \lambda^{-\alpha + \delta} h(\lambda)\, d\lambda}, 
		\end{split}
	\end{equation}
	where $\delta=\frac\epsilon2$  and 
	$\mathcal{P}_n = \{h : \text{$h$ is a probability density on $[0,\pi]$ with $|h|\leq n$}\}$. If $\beta<\al$, we also have $\beta+\delta<\alpha-\delta$ upon choosing $\eps$ small enough, and we can conclude by bounding $\la^{-\beta-\delta}\leq \pi^{\al-\beta-2\delta} \la^{-\al+\delta}$. If $\beta>\al$, take $\eps$ small enough such that $\beta+\delta<1$. Then we can use Lemma \ref{lem:aux-monotone} with $\kappa=\beta-\al+2\delta$ to obtain \eqref{eq:bound-DH}.
\end{proof}

\begin{remark}\label{rem:Dahlhaus}
	Unlike Lemma \ref{lem:toeplitz-norm-0} above, where the upper bound is of order  $n^{(\beta-\alpha)_+/(1-\alpha_+)}$, the lemma on p.\ 80 in \cite{taqqu_toeplitz_1987}  and Lemma~5.3 in \cite{Dahlhaus1989} erroneously presents a smaller upper bound of the order $n^{(\beta-\alpha)_+}$.
	The proof of the latter proceeds along the same lines as above, but it is  claimed that the final supremum in \eqref{eqn:toeplitz-sup} is attained at $h_n(\lambda) = n \bone_{(0,1/n)}(\la)$, leading to the smaller bound. 
	As a counterexample, let $\ov\beta=\beta +\delta= \frac{2}{3}$, $\ov\alpha=\alpha-\delta=\frac{1}{2}$ and $h(\lambda) = n \bone_{(0,n^{-2})}(\la) + n \bone_{(1-n^{-1}+n^{-2},1)}(\la)$. 
	Then the final quotient in \eqref{eqn:toeplitz-sup} is of order $n^{{1}/{3}} = n^{(\ov\beta-\ov\alpha)/(1-\ov\alpha)}$, which is bigger than $n^{1/6}=n^{\ov\beta-\ov\alpha}$.
\end{remark}

Due to the array setting considered in this paper, we need an extension of Lemma~\ref{lem:toeplitz-norm-0} to functions satisfying Assumption~\ref{ass:gauss-lan-2} or Assumption~\ref{ass:gauss-lan-2-ext}.

\begin{lemma}\label{lem:toeplitz-norm}
	For $l=1,\ldots, p$, let $g_l$ and $f_l$ be as in Theorem \ref{thm:toeplitz-inverse}, and suppose furthermore that $f_l> 0$ and $g_l\geq 0$.
	There exists $K=K( \epsilon,L, q, \eta,\balpha)$, uniformly bounded on compacts in $\balpha$ for sufficiently small $\epsilon>0$ such that
	\begin{equation}\label{eq:bound-DH-2}\begin{split}
			\| T_n(g_l)^{\frac{1}{2}} T_n(f_l)^{-\frac{1}{2}}  \|^2_{\op} =  \|  T_n(f_l)^{-\frac{1}{2}} T_n(g_l)^{\frac{1}{2}}  \|^2_{\op} &\leq K  \delta_{n},\\
			\| T_n(g_{l+1})^{\frac{1}{2}} T_n(f_l)^{-\frac{1}{2}}  \|^2_{\op} =  \|  T_n(f_l)^{-\frac{1}{2}} T_n(g_{l+1})^{\frac{1}{2}}  \|^2_{\op} &\leq K  \delta_{n}.
	\end{split}\end{equation}
\end{lemma}
\begin{proof}
	In order to prove the first inequality in \eqref{eq:bound-DH-2}, observe that $g_l = \sum_{i=1}^m g_{l,i}$ with $g_{l,i}\in\Ga_1(d_{i,k},\beta_{i,k},L)$ for every $k$. Therefore, for any $\delta<\frac\eta3$, we have $\lvert g_{l,i}(\la)\rvert \leq L(\delta)d_{i,k'}\lvert\la\rvert^{-\beta_{i,k'}-\delta}$ for every $k$ (recall that $k'=k'(i,k)$), which implies $\lvert g_{l,i}(\la)\rvert \leq L(\delta)\min_k d_{i,k'}\lvert\la\rvert^{-\beta_{i,k'}-\delta}$. Proceeding similarly to \eqref{eqn:toeplitz-sup}, we obtain 
	\begin{align*}
		\| T_n(g_l)^{\frac{1}{2}} T_n(f_l)^{-\frac{1}{2}}  \|_{\op}^2 &\leq \sup_{h\in\calp_n} \frac{\int_{0}^{\pi} g_l(\lambda) h(\lambda)\, d\lambda}{\int_{0}^{\pi} f_l(\lambda) h(\lambda)\, d\lambda} \leq \sup_{h\in\calp_n} \sum_{i=1}^m \frac{\int_{0}^{\pi} g_{l,i}(\lambda) h(\lambda)\, d\lambda}{\int_{0}^{\pi} f_l(\lambda) h(\lambda)\, d\lambda} \\
		&\leq L(\delta) \sup_{h\in\calp_n} \sum_{i=1}^m \frac{\int_{0}^{\pi} \min_k d_{i,k'}\la^{-\beta_{i,k'}-\delta} h(\lambda)\, d\lambda}{\int_{0}^{\pi} f_l(\lambda) h(\lambda)\, d\lambda}.
	\end{align*}
	Next, we have	
	$1/f_l = \sum_{k=1}^q \ov f_{l,k}$ with some  $\ov f_{l,k}\in \Ga_1(1/ \ov c_{i,k},-\ov \al_{i,k},L)$ for all $i$. Therefore, $\lvert\ov f_{l,k}(\la)\rvert\leq L(\delta)\ov c^{-1}_{i',k}\lvert\la\rvert^{\ov \al_{i',k}-\delta}$ for all $i$ (recall that $i'=i'(i,k)$) and we obtain $1/f_l\leq q \max_k  \lvert \ov f_{l,k} \rvert$ and $f_l(\la) \geq q^{-1}\min_k \lvert \ov f_{l,k}(\la)\rvert^{-1}\geq (qL(\delta))^{-1}\min_k  \ov c_{i',k}\lvert\la\rvert^{-\ov \al_{i',k}+\delta}$ for every $i$. Thus,
	\begin{equation*}
		\| T_n(g_l)^{\frac{1}{2}} T_n(f_l)^{-\frac{1}{2}}  \|_{\op}^2 \leq qL(\delta)^2\sup_{h\in\calp_n} \sum_{i=1}^m     \frac{\int_0^\pi \min_k d_{i,k'}\la^{-\beta_{i,k'}-\delta} h(\la)\,d\la}{\int_0^\pi \min_k \ov c_{i',k}\la^{-\ov \al_{i',k}+\delta} h(\la)\,d\la}.
	\end{equation*}
	Let $G_{i}(\la)= \min_k d_{i,k'}\la^{-\beta_{i,k'}-\delta}$ and $F_{i}(\la)=\min_k \ov c_{i',k}\la^{-\ov \al_{i',k}+\delta}$.
	For all $i$, $k$ and $\la\in(0,\pi)$, we have
	\begin{align*}
		G_{i}(\lambda)&= \min_k \left[ d_{i,k'}\la^{-\beta_{i,k'}-\delta}\right]=\min_k \left[\frac{d_{i,k'}}{\ov c_{i',k}}(\ov c_{i',k}\la^{-\ov \al_{i',k}+\delta})\la^{-(\beta_{i,k'}-\ov \al_{i',k}+2\delta)}\right] \\
		&\leq \min_k \left[\ov c_{i',k}\la^{-\ov \al_{i',k}+\delta}\right]\;\max_k\left[ \frac{d_{i,k'}}{\ov c_{i',k}}\la^{-(\beta_{i,k'}-\ov \al_{i',k}+2\delta)}\right] \\
		&\leq F_{i}(\la)\sum_{k=1}^q \frac{d_{i,k'}}{\ov c_{i',k}}\la^{-(\beta_{i,k'}-\ov \al_{i',k}+2\delta)}.
	\end{align*}
	As a consequence,
	\begin{equation}\begin{split}
			\| T_n(g_l)^{\frac{1}{2}} T_n(f_l)^{-\frac{1}{2}}  \|_{\op}^2 &\leq qL(\delta)^2\sup_{h\in\calp_n} \sum_{i=1}^m     \frac{\int_0^\pi G_{i}(\la)h(\la)\,d\la}{\int_0^\pi F_{i}(\la) h(\la)\,d\la}\\
			&\leq qL(\delta)^2 \sum_{i=1}^m \sum_{k=1}^q \frac{d_{i,k'}}{\ov c_{i',k}}   \sup_{h\in\calp_n} \frac{\int_0^\pi F_{i}(\la)\la^{-(\beta_{i,k'}-\ov \al_{i',k}+2\delta)}h(\la)\,d\la}{\int_0^\pi F_{i}(\la) h(\la)\,d\la}.\end{split}\label{eq:sup}
	\end{equation}	
	To bound the right-hand side, we apply Lemma~\ref{lem:aux-monotone} to the function $F_{i}$ and $\kappa=\kappa_{i,k}=\beta_{i,k'}-\ov \al_{i',k}+2\delta$. As in the proof of Lemma~\ref{lem:toeplitz-norm-0}, only $\kappa_{i,k}>0$ is of concern, as the integral fraction is bounded otherwise. 
	Fix $i$ and $k$ with $\kappa_{i,k}>0$  and let $M\geq 2^{1+\kappa_{i,k}}n^{\kappa_{i,k}}$. Clearly, for any $k_\ast$ with $\ov \al_{i'_\ast,k_\ast}+\kappa_{i,k}<1$ (where $i'_\ast = i'(i,k_\ast)$), we have
	\begin{equation}\label{eq:min1} \begin{split}
			n\int_0^{M^{-\frac1{\kappa_{i,k}}}} F_{i}(\la)\la^{-\kappa_{i,k}}\,d\la&	\leq n\ov c_{i'_\ast,k_\ast} \int_0^{M^{-\frac1{\kappa_{i,k}}}} \la^{-\ov \al_{i'_\ast,k_\ast}+\delta-\kappa_{i,k}}\,d\la   \\ &= \frac{n\ov c_{i'_\ast,k_\ast}}{1-\ov \al_{i'_\ast,k_\ast}-\kappa_{i,k}+\delta} M^{1-\frac1{\kappa_{i,k}}(1-\ov \al_{i'_\ast,k_\ast}+\delta)}. \end{split}
	\end{equation}
Such a $k_\ast$ always exists by the first assumption in \eqref{eqn:fg-inverse-choice}.
	At the same time,
	\begin{align}\nonumber
			\frac M4\inf_{\la\in(\frac1{2n},\pi)} F_{i}(\la)&=\frac M4\min_{k}\min_{\la\in(\frac1{2n},\pi)} \ov c_{i',k}\la^{-\ov \al_{i',k}+\delta}\\
			& = \frac M4 \min_{k}  \ov c_{i',k}\left(\pi^{-(\ov \al_{i',k}-\delta)}\bone_{\{\ov \al_{i',k}-\delta\geq0\}}+(2n)^{\ov \al_{i',k}-\delta}\bone_{\{\ov \al_{i',k}-\delta<0\}}\right) \label{eq:min2}\\
			&\geq \frac M{4\pi^{\ga}}\min_{k}  \ov c_{i',k} (2n)^{-(\ov \al_{i',k}-\delta)_-},
	\nonumber\end{align}
	where $\ga=\max_{i,k}  \lvert\ov \al_{i,k}\rvert$.
	Hence, the supremum in \eqref{eq:hsup-res} (with $f=F_{i}$ and $\kappa=\kappa_{i,k}$) is bounded by the value   of $M$ for which the last terms in \eqref{eq:min1} and \eqref{eq:min2} are equal. Denote this value by $\ov M$. Note that $\ov M$ depends on $k_\ast$, which is a free parameter that we can choose, possibly dependent on $n$, $i$ and $k$, as long as   $\ov \al_{i'_\ast,k_\ast}+\kappa_{i,k}<1$ is satisfied.
	
	Let $\un k = \un k(n,i)$ be a minimizer of the last term in \eqref{eq:min2}. If $\un k$ is such that $\ov \al_{\un i',\un k}+\kappa_{i,k}\leq1$ (where $\un i'=i'(i,\un k)$), we choose $k_\ast = \un k$. Then, depending on whether $\ov \al_{\un i',\un k}-\delta\geq0$ or $\ov \al_{\un i',\un k}-\delta<0$, we obtain 
	\[ \ov M = \left(\frac{4\pi^{\ga}}{1-\ov \al_{\un i',\un k}-\kappa_{i,k}+\delta} n\right)^{\frac{\kappa_{i,k}}{1-\ov \al_{\un i',\un k}+\delta}}\] 
	or
	\[\ov M = \left(\frac{4\pi^{\ga}2^{-(\ov \al_{\un i',\un k}-\delta)}}{1-\ov \al_{\un i',\un k}-\kappa_{i,k}+\delta} \right)^{\frac{\kappa_{i,k}}{1-\ov \al_{\un i',\un k}+\delta}}n^{\kappa_{i,k}}.\]
	In any case, it follows that
	\begin{equation}\label{eq:res1}  
		\ov M \leq K n^{\frac{\kappa_{i,k}}{1-(\al'_{i})_{+}}}.
	\end{equation}
	Now suppose that $\ov \al_{\un i',\un k}+\kappa_{i,k}\geq1$ and recall that  $\delta<\frac\eta3$, where $\eta$ is the number from \eqref{eqn:fg-inverse-choice}. In this case, by the condition $\beta_{i,k'}-\ov \alpha_{i',k} < 1-\eta$ in \eqref{eqn:fg-inverse-choice}, we may conclude that $\kappa_{i,k}\leq 1-\eta+2\delta< 1-\frac\eta3$, which implies that $\ov \al_{\un i',\un k}\geq1-\kappa_{i,k}> \frac\eta3$ and therefore, $\ov \al_{\un i',\un k}-\delta>\frac\eta3-\delta>0$. Hence, for  any $k_\ast$ with $\ov \al_{i'_\ast,k_\ast}+\kappa_{i,k}<1$ (there is at least one such $k_\ast$ by the first condition in  \eqref{eqn:fg-inverse-choice}), the bound $\ov M$ is given by
	\[ \ov M = \left(\frac{4\pi^{\ga}}{1-\ov \al_{i'_\ast, k_\ast}-\kappa_{i,k}+\delta}\frac{\ov c_{i'_\ast,k_\ast}}{\ov c_{\un i',\un k}} n\right)^{\frac{\kappa_{i,k}}{1-\ov \al_{i'_\ast, k_\ast}+\delta}}. \]
	As $k_\ast$ was arbitrary, the supremum in \eqref{eq:hsup-res} is bounded by
	\begin{equation}\label{eq:res2} 
		K(R_{i,k}n)^{\frac{\kappa_{i,k}}{1-(\al'_{i})_+}}.
	\end{equation}
	Inserting \eqref{eq:res1} and \eqref{eq:res2} as bounds for the second supremum in \eqref{eq:sup} and using the third condition in \eqref{eqn:fg-inverse-choice} to bound $2\delta/(1-(\ov\al_{i})_+) \leq 2\delta/\eta\leq \eps$ for  $\delta\leq \eps\eta/2$, we obtain the first inequality in \eqref{eq:bound-DH-2}. The second inequality follows by applying the first one to $g_{l+1}$ instead of $g_l$. 
\end{proof}

\subsection{Proof of Theorem \ref{thm:toeplitz-inverse}}\label{sec:proofs-toeplitz}

We need two more lemmas. The first one is a  refinement of Theorem 2 of \cite{Lieberman2004}.
In contrast to the original statement of \cite{Lieberman2004}, we allow for different functions $f_1, f_2,\ldots$ and $g_1,g_2,\dots$, and we explicitly track the effect of the implicit constants. 
This allows us to derive some uniform error bounds if we consider sequences of spectral densities.

\begin{lemma}\label{thm:toeplitz-integral}
	For $l=1,\ldots, p$, let $g_l \in \Gamma_1(c_{g_l}, \alpha_{g_l},L)$ and $f_l \in \Gamma_1(c_{f_l}, \alpha_{f_l},L)$ with $\al_{g_l},\al_{f_l}<1$.
	If $\sum_{l=1}^p (\alpha_{f_l}+\alpha_{g_l})_+< 1-\eta$ for some $\eta\in(0,1)$, 
	then for any $\epsilon>0$ there exists a constant $K=K(\epsilon,p, \eta,\balpha)$ such that
	\begin{equation}\begin{split}
	&	\Biggl| \frac{1}{n} \tr \Biggl( \prod_{l=1}^p T_n(g_l) T_n(f_l) \Biggr) - (2\pi)^{2p-1} \int_{-\pi}^{\pi} \prod_{l=1}^p (g_l(\lambda) f_l(\lambda)) \, d\lambda \Biggr| \\
	&\quad \leq K \Biggl[\prod_{l=1}^p c_{f_l} c_{g_l}\Biggr]  n^{-1+\epsilon + \sum_{l=1}^p(\alpha_{f_l} + \alpha_{g_l})_+}.
		\label{eqn:toeplitz-integral}
		\end{split}
	\end{equation}
	The constant $K$ is bounded on compacts in $\balpha$ for  sufficiently small $\epsilon$.
\end{lemma}

\begin{proof} 
	Let $\delta>0$, to be chosen later, as well as $C_{f_l}=c_{f_l}L(\delta)$ and $C_{g_l}=c_{g_l}L(\delta)$. By assumption,
	\begin{equation*}
		|g_l(\lambda)| + |\lambda  g'_l(\lambda)| \leq C_{g_l} |\lambda|^{-\alpha_{g_l}-\delta},\quad	|f_l(\lambda)| + |\lambda  f'_l(\lambda)| \leq C_{f_l} |\lambda|^{-\alpha_{f_l}-\delta}.
	\end{equation*} 
	Therefore, Theorem 1 in the supplement of \cite{takabatake_note_2022} yields 
	\begin{equation}\begin{split}
		& \left| \frac{1}{n} \tr \left( \prod_{l=1}^p T_n(g_l) T_n(f_l) \right) - (2\pi)^{2p-1} \int_{-\pi}^{\pi} \prod_{l=1}^p (g_l(\lambda) f_l(\lambda)) \, d\lambda \right| \\
		&\quad= {O}\Bigl( n^{-1+\epsilon/2+2p\delta+ \sum_{l=1}^p(\alpha_{f_l}+ \alpha_{g_l})_+} \Bigr)
		\label{eqn:toeplitz-bigO}\end{split}
	\end{equation}
	for any $\eps>0$.
	An inspection of the proof of this result 
	reveals that the implicit constant in \eqref{eqn:toeplitz-bigO} only depends on $C_{f_l}$ and $C_{g_l}$ but not on the other properties of $f_l$ and $g_l$. 
	Hence, choosing $\delta=\eps/(4p)$, we derive \eqref{eqn:toeplitz-integral} with $C_{f_l}$ and $C_{g_l}$ instead of $c_{f_l}$ and $c_{g_l}$. Subsuming $L(\eps/(4p))^{2p}$ into $K$, we obtain \eqref{eqn:toeplitz-integral}.
\end{proof}

			\begin{lemma}\label{lem:toeplitz-integral-sum}
				For $l=1,\ldots, p$, let $g_l$ and $f_l$ be as in Theorem \ref{thm:toeplitz-inverse}. 
				There exists $K=K( \epsilon,L, m, p,q,\eta, \balpha)$ that is uniformly bounded on compacts in $\balpha$ for sufficiently small $\eps$  such that
				\begin{equation} 
					\left| \frac{1}{n} \tr \left( \prod_{l=1}^p T_n(g_l) T_n(f_l^{-1}) \right) - (2\pi)^{2p-1} \int_{-\pi}^{\pi} \prod_{l=1}^p \frac{g_l(\lambda)}{f_l(\lambda)} \, d\lambda \right|  
					\leq \frac{K}{n} (\delta^\ast_{n})^p
					\label{eqn:toeplitz-integral-frac} 
				\end{equation}
			and
				\begin{equation}
					\Biggl| \frac{1}{n} \tr \Biggl( \prod_{l=1}^p T_n(g_l) T_n(f_l)^{-1} \Biggr) - \frac{1}{n} \tr \Biggl( \prod_{l=1}^p T_n(g_l) T_n((4\pi^2 f_l)^{-1} ) \Biggr) \Biggr| \leq \frac{K}{n}  \delta_n^\star    [ ( \delta_n^\star \delta_n )^p +  (\delta_n^\ast )^p ].
					\label{eq:help}
				\end{equation}
			\end{lemma}

			\begin{proof}[Proof of Lemma~\ref{lem:toeplitz-integral-sum}]
				Write $g_l = \sum_{i=1}^m g_{l,i}$ with $g_{l,i}\in  \Gamma_1(d_{i,k^\ast}, \beta_{i,k^\ast},L)$ and
				$1/f_l = \sum_{k=1}^q \ov{f}_{l,k} $ with $\ov{f}_{l,k}\in \Gamma_1(1/\ov c_{i^\ast,k}, -\ov \alpha_{i^\ast,k},L) $.
				By multilinearity and Lemma \ref{thm:toeplitz-integral} (and writing $i^\ast_l = i^\ast(i,k_l)$ and $k^\ast_l=k^\ast(i_l,k)$),
				\begin{align*}
					&  \left| \frac{1}{n} \tr \left( \prod_{l=1}^p T_n(g_l) T_n(f_l^{-1}) \right) - (2\pi)^{2p-1} \int_{-\pi}^{\pi} \prod_{l=1}^p \frac{g_l(\lambda)}{f_l(\lambda)} \, d\lambda \right| \\
					&\quad=\Bigg| \sum_{i_1,\dots,i_p=1}^m \sum_{k_1,\dots, k_p=1}^q \Bigg(\frac{1}{n} \tr \left( \prod_{l=1}^p T_n(g_{l,i_l}) T_n(\ov{f}_{l,k_l}) \right)  - (2\pi)^{2p-1} \int_{-\pi}^{\pi} \prod_{l=1}^p  g_{l,i_l}(\lambda)\ov f_{l,k_l}(\lambda) \, d\lambda\Bigg) \Bigg|\\
					&\quad\leq K n^{-1} \sum_{i_1,\dots,i_p=1}^m \sum_{k_1,\dots, k_p=1}^q \left[\prod_{l=1}^p  \frac{d_{i_l,k^\ast_l}}{\ov c_{i^\ast_l,k_l}} n^{ (\beta_{i_l, k^\ast_l} - \ov \alpha_{i^\ast_l,k_l} )_+ +\epsilon}\right]\\
					&\quad= K n^{-1} \prod_{l=1}^p\left[ \sum_{i=1}^m \sum_{k=1}^q  \frac{d_{i,k^\ast}}{\ov c_{i^\ast,k}} n^{ (\beta_{i, k^\ast} - \ov \alpha_{i^\ast,k} )_+ +\epsilon}\right]= Kn^{-1}  ({\delta}_{n}^\ast)^p,
				\end{align*}
				which is \eqref{eqn:toeplitz-integral-frac}.
				
				To prove \eqref{eq:help},  we may assume without loss of generality that $g_l\geq0$ for all $l$ by decomposing $g_l$ into a positive and a negative part and utilizing multilinearity on the left-hand side. The matrix $T_n(f_l)$ is regular because the upper bound on $1/f_l$ implies $f_l>0$.
				Thus, for $l=1,\ldots, p$, we can define
				\begin{align*}
					A_l &= T_n(g_l)^{1/2}T_n(f_l)^{-1}T_n(g_{l+1})^{1/2}, \\
					B_l &= T_n(g_l)^{1/2}T_n((4\pi^2 f_l)^{-1})T_n(g_{l+1})^{1/2}\quad \text{(with $g_{p+1}=g_1$)},\\
					\Delta_l &= I_n - T_n(f_l)^{1/2} T_n((4\pi^2 f_l)^{-1})T_n(f_l)^{1/2},
				\end{align*}
			which implies
		\begin{equation}\label{eq:diff} 
			A_l-B_l = \left(T_n(g_l)^{1/2} T_n(f_l)^{-1/2}\right)\Delta_l \left(T_n(f_l)^{-1/2} T_n(g_{l+1})^{1/2}\right).
		\end{equation}
				As $\tr(AB)=\tr(BA)$, the left-hand side of \eqref{eq:help} is equal to $\frac1n\left\lvert\tr(\prod_{l=1}^p A_l - \prod_{l=1}^p B_l)\right\rvert$.  
				We can apply \eqref{eqn:toeplitz-integral-frac} with $g_l=f_l$  to find for any $\epsilon>0$,
				\begin{align}\nonumber
					\|\Delta_l\|_F^2 &= \tr(\Delta_l \Delta_l)
					= n-2\,\tr\left(T_n(f_l)T_n((4\pi^2f_l)^{-1})\right)+\tr\left([T_n(f_l)T_n((4\pi^2f_l)^{-1})]^2\right) 
					\\
					&\leq K (\delta^\star_{n})^2.\label{eq:Delta}
				\end{align}				
				Moreover, Lemma~\ref{lem:toeplitz-norm} and the sub-multiplicativity of the operator norm yield $	\|A_l\|_{\op} \leq K \delta_n$.
				We will further show by induction that 
				\begin{equation}\label{eq:ind} 
					\left\lVert\prod_{l=1}^k A_l - \prod_{l=1}^k B_l\right\rVert_F \leq K( \delta^\star_{n} \delta_{n})^p, \quad k=1,\dots,p.
				\end{equation}
			Because $\lVert A\rVert_{\op}\leq \lVert A\rVert_F$, 
			this implies
				\begin{equation}
					\|A_l\|_\op + \|B_l\|_\op \leq 2 \|A_l\|_\op + \|B_l-A_l\|_{\op} \leq K \delta_n^\star \delta_n. \label{eqn:AlBlop}
				\end{equation}
								
				To establish \eqref{eq:ind} for $k=1$, note that by \eqref{eq:diff} and \eqref{eq:Delta} and the inequality $\lVert AB\rVert_F \leq \lVert A\rVert_\op\lVert B\rVert_F$, we have
				\begin{equation*}
					\lVert A_l-B_l\rVert_F \leq\left \lVert T_n(g_l)^{1/2}	T_n(f_l)^{-1/2}\right\rVert_{\op} \lVert \Delta_l\rVert_F \left \lVert T_n(f_l)^{-1/2}	T_n(g_{l+1})^{1/2}\right\rVert_{\op}\leq K\delta^\star_{n}\delta_{n},
				\end{equation*}
			From now on, by induction, we may assume that \eqref{eq:ind} holds for all $k=1,\dots,p-1$. Therefore, together with the bound $\lVert A_l\rVert_{\op} \leq K \delta_{n}$, we obtain 
			\begin{align*} 
						& \left\| \prod_{l=1}^p A_l - \prod_{l=1}^p B_l \right\|_F
						= \left\lVert \sum_{r=1}^p  \left(\prod_{l=1}^{r-1} B_l \right) (A_r-B_r) \left(\prod_{l=r+1}^{p} A_l \right)  \right\rVert_F \\
						&\quad= \left\Vert \sum_{r=1}^p  \left(\prod_{l=1}^{r-1} B_l -\prod_{l=1}^{r-1} A_l\right) (A_r-B_r) \left(\prod_{l=r+1}^{p} A_l \right)+\left(\prod_{l=1}^{r-1} A_l\right)(A_r-B_r) \left(\prod_{l=r+1}^{p} A_l \right)  \right\rVert_F \\
						&\quad\leq\sum_{r=1}^p  \lVert A_r-B_r\rVert_F\prod_{l\neq r}  \lVert A_l\rVert_{\op} \,+\, \sum_{r=1}^p \left\lVert  \prod_{l=1}^{r-1} B_l-\prod_{l=1}^{r-1} A_l \right\rVert_F \lVert A_r-B_r\rVert_F \prod_{l=r+1}^{p}  \lVert A_l\rVert_{\op} \\
					&\quad \leq K\Bigg( \sum_{r=1}^p \delta^\star_{n}\delta_{n}\prod_{l\neq r} \delta_{n}+ \sum_{r=1}^p (  \delta^\star_{n}{\delta_{n}}  )^{r-1}\delta^\star_{n} {\delta_{n}}\prod_{l=r+1}^p {\delta_{n}}\Bigg),
			 \end{align*}
				which implies \eqref{eq:ind} for $k=p$. 
				
				Returning to the trace, if $p\geq2$, we use the bound $\lvert \tr(AB)\rvert \leq \lVert A\rVert_F\lVert B\rVert_F$ and \eqref{eqn:AlBlop} to obtain
				\begin{align*}
					& \left|\tr\left( \prod_{l=1}^p A_l - \prod_{l=1}^p B_l \right)\right|= \left| \tr \left[\sum_{r=1}^p\left(\prod_{l=1}^{r-1} A_l\right)(A_r-B_r) \left(\prod_{l=r+1}^{p} B_l \right)  \right] \right| \\
					&= \left| \tr\left[ \sum_{r=1}^p  \left(\prod_{l=1}^{r-1} A_l -\prod_{l=1}^{r-1} B_l\right) (A_r-B_r) \left(\prod_{l=r+1}^{p} B_l \right)+\left(\prod_{l=1}^{r-1} B_l\right)(A_r-B_r) \left(\prod_{l=r+1}^{p} B_l \right) \right]  \right| \\
					&\leq \sum_{r=1}^p \left\|\prod_{l=1}^{r-1} B_l -\prod_{l=1}^{r-1} A_l\right\|_F \|A_r-B_r\|_F \prod_{l=r+1}^p \|B_l\|_{\op} + \sum_{r=1}^p \left| \tr\left[ \left(\prod_{l=1}^{r-1} B_l\right)(A_r-B_r) \left(\prod_{l=r+1}^{p} B_l \right) \right] \right| \\
					&\leq K \sum_{r=1}^p (\delta_n^\star \delta_n)^{r-1}\delta_n^\star\delta_n (\delta_n^\star\delta_n)^{p-r}   +  \sum_{r=1}^p \left| \tr\left[ \left(\prod_{l=1}^{r-1} B_l\right)(A_r-B_r) \left(\prod_{l=r+1}^{p} B_l \right) \right] \right| \\
					&\leq K  (\delta_n^\star \delta_n)^{p}   +  \sum_{r=1}^p\left| \tr\left[ \left(\prod_{l=1}^{r-1} B_l\right)(A_r-B_r) \left(\prod_{l=r+1}^{p} B_l \right) \right] \right|.
				\end{align*}
			Clearly, the bound in the last line trivially covers the case $p=1$ as well.
				To proceed, define 
				\begin{equation*}
					A_l^* =T_n(g_l) T_n(f_l)^{-1},\quad
					B_l^* =T_n(g_l) T_n((4\pi^2f_l)^{-1}),\quad
					\Delta_l^*= I_n - T_n(f_l) T_n((4\pi^2 f_l)^{-1})
				\end{equation*}
			which implies
			\[A_l^* \Delta_l^* = A_l^*-B_l^*.\]
				Since $\tr(AB)=\tr(BA)$, we can move  one factor of $T_n(g_1)^{1/2}$ from the very left to the very right and obtain
				\begin{align*}
					&\left| \tr\left[ \left(\prod_{l=1}^{r-1} B_l\right)(A_r-B_r) \left(\prod_{l=r+1}^{p} B_l \right) \right] \right| = \left| \tr\left[ \left(\prod_{l=1}^{r-1} B^*_l\right)(A^*_r-B^*_r) \left(\prod_{l=r+1}^{p} B^*_l \right) \right] \right| \\
					&\quad= \left| \tr\left[ \left(\prod_{l=1}^{r-1} B^*_l\right)A^*_r\Delta_r^* \left(\prod_{l=r+1}^{p} B^*_l \right) \right] \right| \\
					&\quad\leq \left| \tr\left[ \left(\prod_{l=1}^{r} B^*_l\right)\Delta_r^* \left(\prod_{l=r+1}^{p} B^*_l \right) \right] \right| 
					+ \left| \tr\left[ \left(\prod_{l=1}^{r-1} B^*_l\right)(A^*_r-B_r^*)\Delta_r^* \left(\prod_{l=r+1}^{p} B^*_l \right) \right] \right|.
				\end{align*}
				The first term is of the order $(\delta_n^*)^p\delta_n^\star$ by virtue of \eqref{eqn:toeplitz-integral-frac}.
				The second term is controlled by
				\begin{align*}
					& \left| \tr\left[ \left(\prod_{l=1}^{r-1} B^*_l\right)(A^*_r-B_r^*)\Delta_r^* \left(\prod_{l=r+1}^{p} B^*_l \right) \right] \right| \\
					&\quad = \left| \tr\left[ \left(\prod_{l=1}^{r-1} B_l\right)\left( T_n(g_r)^{1/2} T_n(f_r)^{-1/2} \right) \Delta_r^2 \left(  T_n(f_r)^{-1/2} T_n(g_{r+1})^{1/2} \right) \left(\prod_{l=r+1}^{p} B_l \right) \right] \right| \\
					&\quad\leq \left(\prod_{l\neq r} \|B_l\|_{\op} \right) \left\| T_n(g_r)^{1/2} T_n(f_r)^{-1/2} \right\|_\op \left\|  T_n(f_r)^{-1/2} T_n(g_{r+1})^{1/2} \right\|_\op  \|\Delta_r\|_F^2 \\
					&\quad\leq (\delta_n^\star \delta_n)^{p-1} \delta_n (\delta_n^\star)^2 
					 =  (\delta_n^\star \delta_n)^{p} \delta_n^\star.
				\end{align*}
				In summary, the term \eqref{eq:help} is of the order 
				\begin{equation*}
					n^{-1}  [ (\delta_n^*)^p \delta_n^\star + (\delta_n^\star \delta_n)^p \delta_n^\star  ] 
				 \leq n^{-1} \delta_n^\star \left( \delta_n^\star \delta_n + \delta_n^\ast\right)^p.\qedhere
				\end{equation*} 
			\end{proof}

			\begin{proof}[Proof of Theorem~\ref{thm:toeplitz-inverse}]
				Combine \eqref{eqn:toeplitz-integral-frac} and \eqref{eq:help} of Lemma~\ref{lem:toeplitz-integral-sum}.
			\end{proof}

\subsection{Proof  of Theorem \ref{thm:CLT}}\label{sec:proofs-quadratic}
\begin{proof}[Proof of Theorem \ref{thm:CLT}]
	Throughout the proof, we denote by $K=K(\epsilon, L, m, p,q,\eta,\balpha)$ a constant which may vary from line to line.
	We introduce 
	\begin{align*}
		A_n(\theta) = T_n(f^{\theta_0}_n)^{\frac{1}{2}} T_n(f^\theta_n)^{-1} \prod_{l=1}^p \left[T_n(g^{\theta}_{n,l}) T_n(f^\theta_n)^{-1} \right] T_n(f^{\theta_0}_n)^{\frac{1}{2}},
	\end{align*}
which is symmetric because of \eqref{eq:symm},
	and note that $X_n = T_n(f^{\theta_0}_n)^\frac{1}{2} Y_n$ where $Y_n \sim \mathcal{N}(0, I_n)$. 
	
	
	We have the representation
	\begin{equation}
			Z_n(\theta_0) 
			=  Y_n^\T A_n(\theta_0) Y_n  - \tr(A_n(\theta_0))\\
			= \sum_{k=1}^n (\epsilon_k^2-1) a_k,
		\label{eqn:Zn-decomp}
	\end{equation}
	where the $\epsilon_k$'s are independent standard normal random  variables and the $a_k$'s are the eigenvalues of $A_n(\theta_0)$. Note that both $\epsilon_k$ and $a_k$ depend on $n$, even though this is not reflected in the notation.
	The random variables $\epsilon_k^2-1$ are iid with zero mean and variance $2$, such that $\Var(Z_n(\theta_0)) = 2\sum_{k=1}^n a_k^2=2\|A_n(\theta_0)\|_F^2$.
	Moreover, the random variables $\epsilon_k^2-1$ are sub-exponential, and hence, by virtue of the Bernstein inequality \citep[Theorem~2.8.2]{Vershynin2003}, there exists some universal $K$ such that $\|Z_n(\theta_0)\|_{\psi_1} \leq K (\sum_{k=1}^n a_k^2)^{1/2} = K \|A_n(\theta_0)\|_F$.
With the notation $\phi_n=\phi_n(\theta_0)$, 
	the central limit theorem then yields $\phi_n^{-1} Z_n(\theta_0)\stackrel{d}{\longrightarrow}\mathcal{N}(0,1)$ if we can show that $\phi_n^{-1} \max_k a_k \to 0$ and $2\phi_n^{-2}\sum_{k=1}^n a_k^2 \to 1$.
	
	Note that $\max_k a_k = \|A_n(\theta_0)\|_{\op}$.
	Using the sub-multiplicativity of the operator norm, we have
	\begin{align*}
		\|A_n(\theta_0)\|_{\op} \leq  \prod_{l=1}^p \|T_n(g^{\theta_0}_{n,l})^\frac{1}{2} T_n(f_n^{\theta_0})^{-\frac{1}{2}}\|_{\op}^2.
	\end{align*}
Assuming $g^{\theta_0}_{n,l}\geq0$  for now, we can use Lemma~\ref{lem:toeplitz-norm} to obtain 
	$\|T_n(g^{\theta}_{n,l})^\frac{1}{2} T_n(f_n^{\theta})^{-\frac{1}{2}}\|_{\op}^2 
	\leq K  \delta_{n}(\theta)$.
	Hence, for any $\epsilon>0$ sufficiently small,
	\begin{align}
		\max_{k=1,\ldots, n} a_k = \|A_n(\theta_0)\|_{\op}
		\leq K \delta_{n}(\theta_0)^p, \label{eqn:An-upper}
	\end{align}	
	which implies $\phi_n^{-1} \max_{k} a_k\to 0$ in the setting of (ii). By splitting $g^{\theta_0}_{n,l}$ into positive and negative parts, one can see that this remains true for general $g^{\theta_0}_{n,l}$. 
	Now note that $\sum_{k=1}^n a_k^2 = \|A_n(\theta_0)\|^2_{F} = \tr(A_n(\theta_0)^2)$. 
	By Theorem \ref{thm:toeplitz-inverse} (and Remark~\ref{rem}), the last term satisfies
	\begin{equation*} 
		\left|\tr(A_n(\theta_0)^2) - \frac{n}{2\pi}\int_{-\pi}^{\pi} \prod_{l=1}^p \left(\frac{g^{\theta_0}_{n,l}(\lambda)}{f_n^{\theta_0}(\lambda)}\right)^2 \,d\lambda \right|  \leq  K \delta_n^\star(\theta_0)  \Bigl(  \delta^\ast_{n}(\theta_0)^{2p}+\delta^\star_{n}(\theta_0)^{2p} \delta_{n}(\theta_0)^{2p} \Bigr). 
	\end{equation*} 
	The assumption $\phi_n^{-1} \delta_n^\star(\theta_0) \big(  \delta^\ast_{n}(\theta_0)^p +\delta^\star_{n}(\theta_0)^p \delta_{n}(\theta_0) ^p\big)\to0$ implies   that we have $2\phi_n^{-2} \sum_{k=1}^n a_k^2=2\phi_n^{-2} \tr(A_n(\theta_0)^2) \to 1$, and we may conclude that $\phi_n^{-1} Z_n(\theta_0)\stackrel{d}{\longrightarrow}\mathcal{N}(0,1)$, establishing {claim (ii)}.
	Moreover, the previous display yields, for any $\epsilon>0$,
	\begin{align*}
		\|Z_n(\theta_0)\|_{\psi_1} \leq  K \|A_n(\theta_0)\|_F 
		\leq    K \phi_n + K \delta_n^\star(\theta_0)\Bigl(\delta^\ast_{n}(\theta_0)^p +\delta^\star_{n}(\theta_0)^{p} \delta_{n}(\theta_0)^p\Bigr),  
	\end{align*}
	which establishes {claim (i)}.
	
	Regarding {claim (iii)}, for arbitrary $\theta\in \Theta_0$, we have
	\begin{equation*}
		Z_n(\theta) 
		=   Y_n^\T A_n(\theta) Y_n  - \tr\left(T_n(f_n^{\theta_0})^{-\frac12}T_n(f_n^{\theta})T_n(f_n^{\theta_0})^{-\frac12}  A_n(\theta)\right) = \tilde{Z}_n(\theta) + \tilde{\Delta}_n(\theta),
	\end{equation*}
where
\begin{align*}
		\tilde{Z}_n(\theta)&=   Y_n^\T A_n(\theta) Y_n  - \tr(A_n(\theta)),\\
		   \tilde{\Delta}_n(\theta)&= \tr\left((I_n-T_n(f_n^{\theta_0})^{-\frac12}T_n(f_n^{\theta})T_n(f_n^{\theta_0})^{-\frac12} ) A_n(\theta)\right). 
	\end{align*}
Note that $Z_n(\theta_0)=\tilde{Z}_n(\theta_0)$.
	To study the non-random term $\tilde{\Delta}_n(\theta)$, we note that
	\begin{align*}
		\tilde{\Delta}_n(\theta)&=\tr\left(T_n(f^{\theta_0}_n)  T_n(f^\theta_n)^{-1} \prod_{l=1}^p \left[T_n(g^{\theta}_{n,l}) T_n(f^\theta_n)^{-1} \right] \right) - \tr\left(  \prod_{l=1}^p \left[T_n(g^{\theta}_{n,l}) T_n(f^\theta_n)^{-1} \right]  \right) \\
		&=\int_0^1 \tr\left(T_n((\nabla f^{\theta +u(\theta_0-\theta)}_n)^\T(\theta_0-\theta)) T_n(f^\theta_n)^{-1} \prod_{l=1}^p \left[T_n(g^{\theta}_{n,l}) T_n(f^\theta_n)^{-1} \right] \right)du.
	\end{align*}
	Apply Theorem \ref{thm:toeplitz-inverse} and property \eqref{eqn:ass-CLT-c} to the integrand to obtain 
	\begin{align*}
		 |\tilde{\Delta}_n(\theta) - J_n^*(\theta) | &=\Biggl|\tilde{\Delta}_n(\theta) 
		-\frac{n}{2\pi}\int_{-\pi}^{\pi} \biggl[ \frac{f_n^{\theta_0}(\lambda)}{f_n^\theta(\lambda)} - 1 \biggr] \prod_{l=1}^p \frac{g_{n,l}^{\theta}(\lambda)}{f_n^\theta(\lambda)}\, d\lambda  \Biggr| \\
		& \leq K\lVert \theta-\theta_0\rVert  \delta^\star_n(\theta) \biggl(\delta^\ast_{n}(\theta)^p  +  \left(\delta_n^\star(\theta) \delta_n(\theta)\right)^{p}
		  \biggr) \sup_{\theta'\in\Theta_0} \delta^\star_n(\theta,\theta') n^\eta\\
		&\leq Kn^{3\eta}  \lVert \theta-\theta_0\rVert\delta^\star_n(\theta)^2\Bigl( \delta^\ast_{n}(\theta)^p +\delta^\star_n(\theta)^{p} \delta_{n}(\theta)^p\Bigr),
	\end{align*}
 where $\sup_{\theta'} \delta^\star_n(\theta,\theta') n^\eta$ accounts for the factor $T_n(\nabla f_n^{\theta+u(\theta_0-\theta)}) T_n(f_n^\theta)^{-1}$ and $\sup_{\theta'} \delta^\star_n(\theta,\theta') \leq K n^{2\eta}\delta^\star_n(\theta)$ due to \eqref{eqn:ass-CLT-c}.  
	
	We now study the stochastic component $\tilde{Z}_n(\theta)$.
	For $\theta_1,\theta_2\in \Theta_0$, we write
	\begin{equation*}
		\tilde{Z}_n(\theta_1) - \tilde{Z}_n(\theta_2) =   Y_n^\T (A_n(\theta_1)-A_n(\theta_2)) Y_n  - \tr(A_n(\theta_1)-A_n(\theta_2)) = \sum_{k=1}^n (\tilde{\epsilon}_k^2-1) \tilde{a}_k, 
	\end{equation*}
	where $\tilde{\epsilon}_k \sim \mathcal{N}(0,1)$ are independent and  $\tilde{a}_k$ are the eigenvalues of $A_n(\theta_1)-A_n(\theta_2)$.
	As we have seen above, Bernstein's inequality shows   that $\|\tilde{Z}_n(\theta_1)-\tilde{Z}_n(\theta_2)\|_{\psi_1}\leq K \|A_n(\theta_1)-A_n(\theta_2)\|_F$ for some $K>0$.
	We have
	\begin{align*}
		\left\| A_n(\theta_1) - A_n(\theta_2) \right\|^2_F 
		&= \sum_{i,j=1}^n \left| (\theta_1 - \theta_2)^\T \int_0^1 \nabla A_n(t \theta_1 + (1-t)\theta_2)_{ij}\, dt \right|^2 \\
		&\leq \|\theta_1-\theta_2\|_2^2 \sum_{r=1}^M  \int_0^1 \sum_{i,j=1}^n  |\partial_{r} A_n(t \theta_1 + (1-t)\theta_2)_{ij}|^2 \, dt \\
		&\leq \|\theta_1-\theta_2\|_2^2 \sum_{r=1}^M \sup_{\theta\in\Theta_0} \|\partial_{r} A_n(\theta)\|_F^2.
	\end{align*}
	Now use the identity $\partial_{r} T_n(f_n^\theta)^{-1} = -T_n(f_n^\theta)^{-1}  T_n(\partial_{r} f_n^\theta) T_n(f_n^\theta)^{-1}$, the product rule of differentiation, the expansion $T_n(f^{\theta_0}_n) = T_n(f^{\theta}_n) + \int_0^1 T_n((\nabla f^{\theta+u(\theta_0-\theta)}_n)^\T(\theta_0-\theta)) du$  and Theorem \ref{thm:toeplitz-inverse} (and Remark~\ref{rem}) to find that
\[ \|\partial_{r} A_n(\theta)\|_F^2
		= \tr[\partial_{r} A_n(\theta)^\T \partial_{r} A_n(\theta) ]= \mathcal{J}_{n,r}(\theta)^2 + R_{n,r}(\theta)\]
		where
			\begin{align*}
			 &\mathcal{J}_{n,r}(\theta)^2\\
			 &= \frac{n}{2\pi}\int_{-\pi}^\pi \left\{ \sum_{s=1}^p \frac{f_n^{\theta_0}(\lambda) }{f_n^\theta(\lambda)} \left(\prod_{l=1,\, l\neq s}^p \frac{g_{n,l}^{\theta}(\lambda)}{f_n^{\theta}(\lambda)}\right)\left( \frac{\partial_{r} g_{n,s}^{\theta}(\lambda)}{f_n^\theta(\lambda)} - \frac{p+1}{p} \frac{g_{n,s}^{\theta}(\lambda)}{f_n^\theta(\lambda)} \frac{\partial_{r} f_n^\theta(\lambda)}{f_n^\theta(\lambda)} \right) \right\}^2\, d\lambda
	\end{align*}
	and
	\begin{align*}
		R_{n,r}(\theta)&\leq K\delta_n^\star(\theta)  \bigg(\sup_{\theta'\in\Theta_0}(1\vee\lVert \theta-\theta_0\rVert\delta^\star_n(\theta,\theta'))\delta^\star_n(\theta)  \delta^\ast_{n}(\theta)^p\\ &\quad+\bigl(1\vee n^{3\eta}\delta^\star_n(\theta)\lVert\theta-\theta_0\rVert\bigr) \delta^\star_n(\theta)^{p+1}n^{2\eta} \delta_{n}(\theta)^p\bigg)^2 \\
		&\leq Kn^{10\eta}(1\vee\delta^\star_n(\theta)\lVert \theta-\theta_0\rVert)\delta^\star_n(\theta)^2\Big( \delta^\ast_{n}(\theta)^p +  \delta^\star_n(\theta)^{p}\delta_{n}(\theta)^p\Big)^2.
	\end{align*}
	Hence, we have shown that
	\begin{equation}\label{eqn:chaining}\begin{split}  
		\|\tilde{Z}_n(\theta_1)-\tilde{Z}_n(\theta_2)\|_{\psi_1}
		&\leq K\left\| A_n(\theta_1) - A_n(\theta_2) \right\|_F    \\
		&  \leq K \|\theta_1-\theta_2\| \sup_{\theta\in\Theta_0} \bigg(  \max_{r}\mathcal{J}_{n,r}(\theta)
		+n^{5\eta}(1\vee\delta^\star_n(\theta)\lVert \theta-\theta_0\rVert) \delta^\star_n(\theta)^2  \\
		&\quad\times\Big(\delta^\ast_{n}(\theta)^p+ \delta^\star_n(\theta)^{p}  \delta_{n}(\theta)^p\Big)\bigg) \\
		&= \|\theta_1-\theta_2\| \widetilde{K}_n, \end{split} 
\end{equation}
	where $\widetilde{K}_n$ is the supremum term.

	Since $\Theta_0$ is a bounded $M$-dimensional set with diameter $\text{diam}(\Theta_0)=\sup_{\theta\in\Theta_0}\|\theta-\theta_0\|$, the exponential inequality \eqref{eqn:chaining} may be leveraged by a standard chaining argument using the metric $d(\theta_1,\theta_2) = \|\theta_1-\theta_2\|\widetilde{K}_n$.
	This yields the covering numbers $N(z,\Theta_0,d) \propto  (\widetilde{K}_n\cdot \mathrm{diam}(\Theta_0) / z)^{M}$, and allows us to establish the uniform bound
	\begin{align*}
		\left\|\sup_{\theta\in\Theta_0} |\tilde{Z}_n(\theta)- \tilde{Z}_n(\theta_0)|\right\|_{\psi_1} 
		& \leq C \int_0^{\widetilde{K}_n  \mathrm{diam}(\Theta_0)} \log\left( \frac{z}{\widetilde{K}_n  \mathrm{diam}(\Theta_0)}\right)^{-M} \, dz  \\
		&= C \widetilde{K}_n \mathrm{diam}(\Theta_0) \int_0^{1} \lvert\log z\rvert  \, dz   \\
		& = C \widetilde{K}_n \mathrm{diam}(\Theta_0) ,
	\end{align*}
	for a potentially bigger $C$ at each occurrence. 
	
	To simplify the expression for $\mathcal{J}_{n,r}(\theta)$, observe that \eqref{eqn:ass-CLT-c} implies the bound $f_n^{\theta_0}(\la)/f_n^\theta(\la) \leq L(\eps)^2\sum_{i,k} c_{i,k'}(n,\theta_0)\ov c_{i',k}(n,\theta)^{-1}\lvert \la\rvert^{\ov\al_{i',k}(\theta)-\al_{i,k'}(\theta_0)-2\eps}\leq Kn^{2\eta} L(\eps^2)mq\lvert\la\rvert^{-2\eps-\eta}$ as well as the bound
	$\lvert \partial_r f_n^{\theta}(\la)\rvert/f_n^\theta(\la) \leq K  L(\eps^2)n^\eta mq\lvert\la\rvert^{-2\eps}$. Therefore, 
	\begin{align*}
	\mathcal{J}_{n,r}(\theta)^2 &\leq K n\int_{-\pi}^\pi \left\{  \sum_{s=1}^p \frac{f_n^{\theta_0}(\lambda) }{f_n^\theta(\lambda)} \left|\prod_{l=1,\, l\neq s}^p \frac{g^{\theta}_{n,l}(\lambda)}{f_n^{\theta}(\lambda)}\right|\left( \frac{|\partial_{r} g_{n,s}^{\theta}(\lambda)|}{f_n^\theta(\lambda)} + \frac{|g^{\theta}_{n,s}(\lambda)|}{f_n^\theta(\lambda)} \frac{|\partial_{r} f_n^\theta(\lambda)|}{f_n^\theta(\lambda)} \right)\right\}^2 \, d\lambda \\
		&  \leq K n^{1+6\eta}\int_{-\pi}^\pi |\lambda|^{-2\eta-8\epsilon} \left\{ \sum_{s=1}^p  \left|\prod_{l=1,\, l\neq s}^p \frac{g^{\theta}_{n,l}(\lambda)}{f_n^{\theta}(\lambda)}\right| \frac{|g^{\theta}_{n,s}(\lambda)| + |\partial_{r} g_{n,s}^{\theta}(\lambda)|}{f_n^\theta(\lambda)}\right\}^2 \, d\lambda \\
		&  \leq K n^{1+6\eta}\int_{-\pi}^\pi |\lambda|^{-3\eta} \left\{ \prod_{l=1}^p \frac{|g^{\theta}_{n,l}(\lambda)| + \| \nabla g^{\theta}_{n,l}(\lambda)\|}{f_n^{\theta}(\lambda)} \right\}^2 \, d\lambda,
	\end{align*}
	upon choosing $\epsilon>0$ small enough. 
\end{proof}

\subsection{Proof of Theorem \ref{thm:LAN}}

We proceed to study the likelihood expansion of $X_n$.
The log-likelihood $l_n(\theta)$ is given by
\begin{align*}
	l_n(\theta) = -\frac{n}{2} \log (2\pi) -\frac{1}{2} \log [\det T_n(f_n^{\theta})] - \frac{1}{2}   X_n^\T T_n(f_n^{\theta})^{-1} X_n.
\end{align*}
The derivation of the LAN property for this model is based on a second-order Taylor expansion of $l_n(\theta)$.
To compute its first two derivatives, we write $\partial_k = \partial_{\theta_k}$ as before and note that $\partial_k T_n(f_n^\theta) = T_n (\partial_k f_n^\theta)$.
Furthermore, we use that for a parametric symmetric matrix $A=A(x)$,
\begin{align*}
		\partial_x \tr(A) &= \tr(\partial_x A),&
		\partial_x A^{-1} &= -A^{-1} (\partial_x A) A^{-1},\\
		\partial_x \det(A)&= \det(A) \tr(A^{-1} \partial_x A),&
		\partial_x \log(\det(A)) &= \tr(A^{-1} \partial_x A).
\end{align*}
For the log-likelihood, we thus obtain the expressions
\begin{equation}
	\partial_{k} l_n(\theta) 
	= -\frac{1}{2} \tr \left( T_n(f_n^{\theta})^{-1}  T_n(\partial_{k}f_n^{\theta}) \right) + \frac{1}{2}  X_n^\T T_n(f_n^{\theta})^{-1}  T_n(\partial_{k} f_n^{\theta}) T_n(f_n^{\theta})^{-1} X_n ,  \label{eqn:loglik-gradient}\end{equation}
and
\begin{align}
		\partial_{kj} l_n(\theta) 
		&= -\frac{1}{2} \tr \left( T_n(f_n^{\theta})^{-1}  T_n(\partial_{kj}f_n^{\theta}) \right) + \frac{1}{2} \tr \left( T_n(f_n^{\theta})^{-1} T_n(\partial_jf_n^{\theta}) T_n(f_n^{\theta})^{-1}  T_n(\partial_{k}f_n^{\theta}) \right)\nonumber\\
		&\quad + \frac{1}{2}   X_n^\T T_n(f_n^{\theta})^{-1} T_n(\partial_{kj} f_n^{\theta}) T_n(f_n^{\theta})^{-1} X_n\label{eqn:loglik-hess} \\
		&\quad - X_n^\T T_n(f_n^{\theta})^{-1} T_n(\partial_j f_n^{\theta})T_n(f_n^{\theta})^{-1}  T_n(\partial_{k}f_n^{\theta}) T_n(f_n^{\theta})^{-1} X_n. 
	\nonumber 
\end{align}

\begin{proof}[Proof of Theorem \ref{thm:LAN}]
	As remarked after Assumption~\ref{ass:gauss-lan-2-ext}, one can view Assumption~\ref{ass:gauss-lan-2} as a special case of Assumption~\ref{ass:gauss-lan-2-ext} except that the growth condition $\frac{1}{\sqrt{Kn^{r}}}\leq\ov c_{i,k}(n,\theta)\leq \sqrt{Kn^{r}}$ is not assumed. This condition is used below only to show that $R_{i,k}(n,\theta)$   in \eqref{eq:R-2} is of polynomial growth in $n$. This is trivially true under Assumption~\ref{ass:gauss-lan-2} (in fact, we have $R_{i,k}(n,\theta)=1$ in this case). Therefore, there is no loss of generality to prove the theorem under Assumption~\ref{ass:gauss-lan-2-ext}, which we do now.
	
	Let $a\in\R^M$, $p=1$ and $g_n^{\theta}(\lambda)=g_{n,1}^{\theta}(\lambda) = \frac12a^\T R_n^\T \nabla f_n^{\theta}(\lambda)$ such that   $a^\T R_n^\T \nabla l_n(\theta_0) = Z_n(\theta_0)$ where $Z_n(\theta_0)$ is as in Theorem \ref{thm:CLT}.
	By Assumption \ref{ass:gauss-lan-2-ext}, we have $\delta^\ast_{n}(\theta)\leq Kn^\eps\lVert R_n\rVert^{1/2+\iota}$  as well as $\delta_{n}(\theta)\leq K\lVert R_n\rVert n^{\eta+\eps(1+r)}$. Upon  choosing $i^\star=i^\ast$ and $k^\star = k^\ast$, we also have $\delta^\star_n(\theta)\leq Kn^\eps \lVert R_n\rVert^{-1/2+\iota}$.
	As $\lVert R_n\rVert$ decays to $0$ polynomially, one can choose $\eps<\eta/(1+r)$ and $\eta\in(0,1)$ small enough such that   $\lVert R_n\rVert^\iota n^{2\eta}\to0$. In this case, we obtain
	\begin{align*}
	 \delta_n^\star(\theta_0)\left(\delta^\ast_{n}(\theta_0)+\delta^\star_n(\theta_0)\delta_{n}(\theta_0) \right)
		\leq Kn^{4\eta}\lVert R_n\rVert^{2\iota}\to0.
	\end{align*}
	Moreover, by Assumption \ref{ass:gauss-lan-1}, $\phi_n(\theta_0)\to a^\T I(\theta_0)a >0$. Thus, the conditions of Theorem \ref{thm:CLT} (ii) are satisfied and $a^\T R_n^\T\nabla l_n(\theta_0) \wconv \mathcal{N}(0, a^\T I(\theta_0)a)$. 
	As $a\in \R^M$ was arbitrary, the Cramer--Wold device yields $\xi_n(\theta_0)=R_n^\T \nabla l_n(\theta_0)\wconv \mathcal{N}(0,I(\theta_0))$.  
	
	Next, 	choose some $a\in\mathcal{H}$.
	Since $\mathcal{H}$ is compact and $\log(n)R_n\to 0$, for each $\eta>0$, there is some $n(\eta)$ such that for all $n\geq n(\eta)$, we have $\theta_0 + \log(n)R_n a \in B(\eta)\subset\Theta$.
	For $n\geq n(\eta)$, the mean value theorem implies that the log-likelihood may  be represented as 
	\begin{equation}
		l_n(\theta_0+R_n a)-l_n(\theta_0) 
		= a^\T R_n^\T \nabla l_n(\theta_0)	+ \frac{1}{2}a^\T R_n^\T \tilde{I}_n R_n a, 
		\label{eqn:ln-expansion}
	\end{equation}
	where $\tilde{I}_n
	= D^2 l_n(\theta_0 + z R_n a)$ 
	for some $z\in[0,1]$. 
	We want to show  that 
	\begin{align}
	\sup_{b\in \mathcal{H}} \left| d^\T R_n^\T[D^2 l_n(\theta_0+\log(n)R_n b)] R_n d + d^\T I(\theta_0)d \right| \pconv 0, \label{eqn:fisher-uniform}
\end{align}
	for any $d\in\R^M$ with $\lVert d\rVert\leq1$ and any compact set $\mathcal{H}\subset \R^M$. This implies that   $d^\T R_n^\T[D^2 l_n(\theta_0+R_n b)] R_n d \pconv I(\theta_0)$ uniformly on compacts in $b\in\mathcal{H}$, which in turn yields $R_n^\T\tilde{I}_n R_n \pconv I(\theta_0)$, uniformly in $a\in\mathcal{H}$.
	
To prove \eqref{eqn:fisher-uniform}, define $h_{n}^{\theta} =  d^\T R_n^\T \nabla f_n^\theta$ and $\ov{h}_{n}^{\theta} =  d^\T R_n^\T D^2 f_n^\theta R_n d$.
	Then, using the representation \eqref{eqn:loglik-hess} for $D^2l_n(\theta)$, we obtain
	\begin{equation}
		d^\T R_n^\T D^2 l_n(\theta) R_n d 
		= \frac{1}{2} Z_{n}^1(\theta) - Z_{n}^2(\theta) - \frac{1}{2} \tr\left( T_n(f^\theta_n)^{-1} T_n(h_{n}^{\theta})T_n(f^\theta_n)^{-1} T_n(h_{n}^{\theta})  \right) , \label{eqn:fisher-uniform-decomp} 
	\end{equation}
where
\begin{align*}
		Z_{n}^1(\theta) 
		&=  X_n^\T T_n(f^\theta_n)^{-1} T_n(\ov{h}_{n}^{\theta})T_n(f^\theta_n)^{-1} X_n   - \tr\left( T_n(f^\theta_n)^{-1} T_n(\ov{h}_{n}^{\theta}) \right), \\
		Z_{n}^2(\theta) 
		&=   X_n^\T T_n(f^\theta_n)^{-1} T_n(h_{n}^{\theta})T_n(f^\theta_n)^{-1} T_n(h_{n}^{\theta})T_n(f^\theta_n)^{-1} X_n   \\
		&\quad - \tr\left( T_n(f^\theta_n)^{-1} T_n(h_{n}^{\theta})T_n(f^\theta_n)^{-1} T_n(h_{n}^{\theta}) \right). 
	\end{align*}

Regarding $Z^1_n(\theta)$,	we apply Theorem \ref{thm:CLT} (iii) with $p=1$ and $g_{n,1}^{\theta}=\ov h_n^\theta$. By Assumption~\ref{ass:gauss-lan-2-ext}, choosing $i^\star=i^\ast$ and $k^\star=k^\ast$, we obtain for $\eps<\eta/(1+r)$ that $\delta^\ast_{n}(\theta)\leq K\lVert R_n\rVert^{3/2+\iota}n^\eta$, $\delta^\star_n(\theta)\leq K\lVert R_n\rVert^{-1/2+\iota}n^\eta$ as well as $\delta_{n}(\theta)\leq Kn^{2\eta}\|R_n\|^2$.
Thus, if $\|\theta-\theta_0\|\leq K \|R_n\| \log (n)$, then
\begin{equation*}
	n^{5\eta} \left(1\vee \delta_n^\star(\theta) \|\theta-\theta_0\| \right)\delta_n^\star(\theta_0)^2\left(\delta^\ast_{n}(\theta_0)+\delta^\star_n(\theta_0)\delta_{n}(\theta_0) \right) \leq K\lVert R_n\rVert^{1/2+4\iota}n^{10\eta}\to0.
\end{equation*}
Moreover, $J_n(\theta)\leq K n^{- {1}/{2}+3\eta/2} \|R_n\|$ by Assumption \ref{ass:gauss-lan-1}, and hence $J_n(\theta)\leq K n^{-{1}/{2}-4\eta}$ for $\eta$ small enough.
Analogously, $\phi_n(\theta_0)^2\leq K n^{\eta} \|R_n\|^2 \leq K n^{-\eta}$. 
Thus, Theorem \ref{thm:CLT} (iii) yields for sufficiently small $\eta$ that
\begin{align*}
	\sup_{b\in \mathcal{H}} \left| Z_n^1(\theta_0 + \log(n) R_n b) - J_n^*(\theta_0+\log(n)R_n b)  \right| \pconv 0.
\end{align*}
Furthermore, observe that for $\theta_n(b)=\theta_0+ \log(n) R_n b$ and sufficiently small $\eta$,
\begin{align*}
	J_n^*(\theta_n(b)) 
	&\leq K n \log (n)\sup_{c\in \mathcal{H}} \|b-c\| \int_{-\pi}^\pi \frac{\|R_n^\T \nabla f_n^{\theta_n(c)}(\lambda)\|   \| \ov{h}_n^{\theta_n(b)}(\lambda) \| }{f_n^{\theta_n(b)}(\lambda)^2}\, d\lambda \\
	&\leq K n \log (n) \| R_n\| \sup_{\theta,\theta'\in B(\eta)} \int_{-\pi}^\pi \frac{\|R_n^\T \nabla f_n^{\theta'}(\lambda)\|   \| R_n^\T D^2f_n^{\theta}(\lambda) \| }{f_n^{\theta}(\lambda)^2}\, d\lambda \\
	&\leq K n^{1+4\eta} \log (n) \| R_n\|  \int_{-\pi}^\pi |\lambda|^{-\eta} \frac{\|R_n^\T \nabla f_n^{\theta_0}(\lambda)\|   \| R_n^\T D^2f_n^{\theta_0}(\lambda)\| }{f_n^{\theta_0}(\lambda)^2}\, d\lambda \\
	&\leq K n^{-\eta}, 
\end{align*}
where we used Assumption~\ref{ass:gauss-lan-2-ext} for the third step and the Cauchy--Schwarz inequality and Assumption \ref{ass:gauss-lan-1} for the last step.
Hence, for any compact set $\mathcal{H}\subset\R^M$, we find that
\begin{align*}
	\sup_{b\in \mathcal{H}} \left| Z_n^1(\theta_0 + \log(n) R_n b) \right| \pconv 0.
\end{align*} 
	
Regarding $Z^2_n(\theta)$, 
	we apply Theorem \ref{thm:CLT} (iii) with $p=2$ and $g^\theta_{n,1}=g^\theta_{n,2}=h_n^\theta$. Note that $\delta_{n}(\theta)\leq Kn^{2\eta}\|R_n\|$, $\delta^\ast_{n}(\theta)\leq Kn^{\eta}\lVert R_n\rVert^{1/2+\iota}$ and $\delta_n^\star(\theta)\leq K\lVert R_n\rVert^{-1/2+\iota}n^\eta$. 
	Therefore,  for $\|\theta-\theta_0\|\leq K \|R_n\| \log(n)$, we find 
	\begin{equation*}
		n^{5\eta}\left(1\vee \delta_n^\star(\theta) \|\theta-\theta_0\| \right)\delta_n^\star(\theta_0)^2\left(\delta^\ast_{n}(\theta_0)^2+\delta^\star_n(\theta_0)^2\delta_{n}(\theta_0)^2 \right)  \leq Kn^{13\eta} \lVert R_n\rVert^{4\iota} \to 0
	\end{equation*} 
for small enough $\eta$. 
	Moreover, using Assumption~\ref{ass:gauss-lan-2-ext} for the second step and Assumption~\ref{ass:gauss-lan-1} for the third step, we obtain for $\epsilon,\eta>0$ sufficiently small,
	\begin{align*}
		\phi_n(\theta_0)^2	&\leq  \frac n\pi\int_{-\pi}^{\pi} \|R_n\|^2  \left(\frac{\lVert \nabla f_n^{\theta_0}(\lambda)\rVert}{f_n^{\theta_0}(\lambda)}\right)^2  \left(\frac{ \lVert R_n^\T\nabla f_n^{\theta_0}(\lambda)\rVert}{f_n^{\theta_0}(\lambda)}\right)^2\, d\lambda \\
&		\leq K \|R_n\|^2n\int_{-\pi}^{\pi}  n^{2\eta} |\lambda|^{-4\epsilon} \left(\frac{\lVert R_n^\T\nabla f_n^{\theta_0}(\lambda)\rVert}{f_n^{\theta_0}(\lambda)}\right)^2 \, d\lambda 
		\leq K n^{3\eta} \|R_n\|^2 
		\leq K n^{-\eta}.
	\end{align*}
	Assumption \ref{ass:gauss-lan-1} also yields
	$\ov J_n=\sup_{\theta\in B(\eta)} J_n(\theta) \leq K n^{-1/2+5\eta/2}\lVert R_n\rVert$ for all $\theta\in B(\eta)$ and $\ov J_n^\ast = \sup_{b\in\calh} J_n^\ast(\theta_0+\log(n)R_nb)\leq Kn^{\eta}\log (n)\lVert R_n\rVert$.
	Since $R_n\to 0$ polynomially, for $n$ large enough such that $\theta_0+\log(n)R_nb\in B(\eta)$ for all $b\in\mathcal{H}$, we have
$
		\sup_{b\in\mathcal{H}} |Z_n^2(\theta_0+\log(n) R_n b)| \stackrel{\P}{\longrightarrow} 0
$
	upon choosing $\eta$ small enough.
	
 Returning to the decomposition \eqref{eqn:fisher-uniform-decomp}, it remains to  the study the third term.
	We apply Theorem~\ref{thm:toeplitz-inverse}  with $p=2$ and  $g^\theta_{n,1}=g^\theta_{n,2}=h_n^\theta$. As in the previous step, we have $\delta_{n}(\theta)\leq K n^{2\eta}\lVert R_n\rVert$, $\delta^\ast_{n}(\theta)\leq Kn^{\eta}\lVert R_n\rVert^{1/2+\iota}$ and $\delta^\star_n(\theta)\leq K\lVert R_n\rVert^{-1/2+\iota}n^\eta$.
	Thus,  
	\begin{align*}
		&\sup_{\theta\in B(\eta)} \left|\frac{1}{2}\tr\left( T_n(f^\theta_n)^{-1} T_n(h_{n}^{\theta})T_n(f^\theta_n)^{-1} T_n(h_{n}^{\theta})  \right) - \frac{n}{4\pi}\int_{-\pi}^{\pi} \frac{h_n^{\theta}(\lambda)^2}{f_n^\theta(\lambda)^2}  \,d\la\right|\\
		&\quad\leq K (n^{3\eta}\lVert R_n\rVert^{1/2+2\iota}+ n^{7\eta}\lVert R_n\rVert^{1/2+3\iota}),
	\end{align*}
	which tends to zero as $n\to\infty$ for small enough $\eta$.
	Finally, by Assumption \ref{ass:gauss-lan-1}, 
	\begin{align*}
		&\quad \sup_{\substack{\theta = \theta_0 + \log(n)R_n b, \\ b\in\mathcal{H}}} \left|\frac{n}{4\pi}\int_{-\pi}^{\pi} \frac{h_n^{\theta}(\lambda)^2}{f_n^\theta(\lambda)^2} \,d\la- d^\T I(\theta_0)d  \right| 
		\to 0
	\end{align*}
as $n\to\infty$. This proves \eqref{eqn:fisher-uniform}, which in turn yields \eqref{eq:thm.1}.
	
Concerning \eqref{eq:thm.2}, note that 
	instead of solving $\nabla l_n(\theta)=0$ for $\theta$, we may equivalently solve $F_n(a)=R_n^\T\nabla l_n(\theta_0+R_n a)=0$ in $a\in\R^M$.
	Above, we have shown that $F_n(0)\wconv \mathcal{N}(0, I(\theta_0))$ and 
$
		\sup_{\|a\|\leq \log(n)}\left\| DF_n(a) + I(\theta_0) \right\| \pconv 0.
$
	Standard results for estimating equations (e.g., Theorem~A.2 of \cite{mies_estimation_2023}) yield the existence of a sequence $\widehat{a}_n$ of random vectors such that $\P^n_{\theta_0}(F_n(\widehat{a}_n)=0)\to 1$ and $\widehat{a}_n \wconv \mathcal{N}(0, I(\theta_0)^{-1})$.
	Equivalently, there is  a sequence $\widehat{\theta}_n=\theta_0 + R_n \widehat{a}_n$ such that $\P^n_{\theta_0}(\nabla l_n(\widehat{\theta}_n)=0)\to 1$ and $R_n^{-1}(\widehat{\theta}_n-\theta_0)\wconv \mathcal{N}(0, I(\theta_0)^{-1})$.
\end{proof}

	\section*{Acknowledgment}
	We thank Shuping Shi and Jun Yu for valuable comments. CC is partially supported by ECS Grant 26301724.

	\bibliography{fbm}
	\bibliographystyle{abbrvnat}

\end{document}